\documentclass[12pt,twoside,reqno]{amsart}

\usepackage{version}         

 \usepackage{color}

\usepackage{multirow}
\usepackage[OT1]{fontenc}
\usepackage{type1cm}
\usepackage{amssymb}
\usepackage{mathrsfs}
\usepackage[english]{babel}
\usepackage[latin1]{inputenc}
\usepackage[T1]{fontenc}
\usepackage{amsfonts}
\usepackage{amsmath}
\usepackage{amssymb}
\usepackage{amsthm}
\usepackage{bbm}
\usepackage{extarrows}
\usepackage{graphicx}
\usepackage[usenames,dvipsnames]{xcolor}
\usepackage{wrapfig}
\usepackage{type1cm}
\usepackage[bf]{caption}
\usepackage{esint}
\usepackage{version}
\usepackage[colorlinks = true, citecolor = black, linkcolor = black, urlcolor = black, pdfstartview=FitH]{hyperref}
\usepackage{caption}
\usepackage{tikz}
\usepackage{bbding, wasysym}
\usepackage{enumitem}

\usepackage[left=2.3cm,top=3cm,right=2.3cm]{geometry}
\numberwithin{equation}{section}

\theoremstyle{plain}
\newtheorem{theorem}{Theorem}[section]

\newtheorem{corollary}[theorem]{Corollary}
\newtheorem{proposition}[theorem]{Proposition}
\newtheorem{lemma}[theorem]{Lemma}

\theoremstyle{remark}
\newtheorem{remark}[theorem]{Remark}

\newtheorem{example}[theorem]{Example}
\newtheorem*{ack}{Acknowledgement}

\theoremstyle{definition}
\newtheorem{definition}[theorem]{Definition}

\newcommand{\R}{\mathbb{R}}

\newcommand{\N}{\mathbb{N}}

\renewcommand{\P}{\mathbb{P}}

\newcommand{\cB}{\mathcal{B}}
\newcommand{\cM}{\mathcal{M}}
\newcommand{\cH}{\mathcal{H}}
\newcommand{\cN}{\mathcal{N}}

\newcommand{\cF}{\mathcal{F}}

\newcommand{\cG}{\mathcal{G}}
\newcommand{\cC}{\mathcal{C}}

\newcommand{\cW}{\mathcal{W}}

\newcommand{\cP}{\mathcal{P}}

\newcommand{\cA}{\mathcal{A}}

\renewcommand{\emptyset}{\varnothing}
\renewcommand{\epsilon}{\varepsilon}
\renewcommand{\rho}{\varrho}
\renewcommand{\phi}{\varphi}

\renewcommand{\i}{\mathtt{i}}

\begin{document}

\title{Random walks associated to beta-shifts}

\author{Bing Li}
\address{School of Mathematics \\
         South China University of Technology \\
         Guangzhou, 510641 \\
         P.R.\ China}
\email{scbingli@scut.edu.cn}

\author{Yao-Qiang Li}
\address{School of Mathematics \\
         South China University of Technology \\
         Guangzhou, 510641 \\
         P.R.\ China}
\email{scutyaoqiangli@qq.com}

\address{Institut de Math\'ematiques de Jussieu - Paris Rive Gauche \\
         Sorbonne Universit\'e - Campus Pierre et Marie Curie\\
         Paris, 75005 \\
         France}
\email{yaoqiang.li@imj-prg.fr}

\author{Tuomas Sahlsten}
\address{School of Mathematics \\
         University of Manchester \\
         Manchester, M13 9PL \\
         UK}
\email{tuomas.sahlsten@manchester.ac.uk}

\date{\today}

\begin{abstract}
We study the dynamics of a simple random walk on subshifts defined by the beta transformation and apply it to find concrete formulae for the Hausdorff dimension of digit frequency sets for $\beta>1$ that solves $\beta^{m+1}-\beta^m-1=0$ generalising the work of Fan and Zhu. We also give examples of $\beta$ where this approach fails.
\end{abstract}

\maketitle

\section{Introduction}

Let $\Sigma = \{0,1\}^\N$ be the full shift and let $\Sigma^*$ be the set all finite words. Then any closed shift invariant subset of $\Sigma$ is called a \textit{subshift}. For any subshift of $\Sigma$ we can always write them as a set $\Sigma_\cW$ for some subset $\cW \subset \Sigma^*$ by removing all the sequences from $\Sigma$ containing substrings from $\cW$. The set $\cW$ is called the collection of all forbidden words. If $\cW$ is finite, then $\Sigma_\cW$ is called a subshift of finite type.

The main example in this paper we consider is the subshift $\Sigma_\beta \subset \Sigma$ defined by the possible \textit{$\beta$-expansions} $w_1w_2\dots$ to
$$x = \sum_{j = 1}^\infty w_j \beta^{-j}$$
of real numbers $x$, for $\beta > 1$, where the digits $w_j \in \{0,1\}$ are obtained by the natural filtration of $[0,1]$ defined by the \textit{$\beta$-transformation} $T_\beta(x) = \beta x \mod 1$ on $[0,1]$. For example in the case $\beta$ is the Golden ratio, then $\Sigma_\beta = \Sigma_{\{11\}}$ with forbidden word $11$. These expansions were introduced by R\'enyi \cite{Ren57} in 1957 and they have since been of wide interest throughout metric number theory and fractal geometry, and in analog-to-digital signal conversions in the study beta-encoders \cite{Ward}.

The algebraic properties of the number $\beta$ link deeply to the dynamical properties of the subshift $\Sigma_\beta$, for example, a classical result of Parry \cite{Pa60} says is that $\Sigma_\beta$ is a subshift of finite type if and only if $\beta$ is a simple number, that is, $1$ has a finite $\beta$-expansion. In this paper we will study further dynamical characterisations of $\Sigma_\beta$ from the point of view of random walks on the finite words $\Sigma_\beta^*$ associated to $\Sigma_\beta$.

Let $\cW$ be any set of forbidden words of the full shift $\Sigma$. Given $0 < p < 1$, there is a natural biased random walk $X_n = \omega_1\omega_2\dots \omega_n$ on $\Sigma^*$ for random variables $\omega_1,\omega_2,\dots \in \{0,1\}$ defined as follows. If $X_{n-1} = w \in \Sigma_\cW^{n-1}$, where $w1 \notin \cW$, then the probability of $\omega_n = 0$ is $p$ and $\omega_n = 1$ by $1-p$ respectively. If $w1 \in \cW$, then the probability of $\omega_n = 0$ is $1$. The random walk $(X_n)$ defines a probability distribution $\mu_p$ supported on the subshift $\Sigma_\cW$ by setting
$$\mu_p[w] := \P(X_{|w|} = w)$$
for all $w \in \Sigma^*$ and cylinder $[w]$. Then $\mu_p[0] = p$, $\mu_p[1] = 1-p$, and if $w1 \notin \cW$, we have $\mu_p[w0] = p\mu_p[w]$ and $\mu_p[w1] = (1-p)\mu_p[w]$. If $w1 \in \cW$, we have $\mu_p[w0] = \mu_p[w]$. Then $\mu_p$ defines a natural probability measure $\nu_p = \pi_\beta \mu_p$ on $[0,1]$ under the natural projection
$$\pi(w) = \sum_{j=1}^\infty w_j \beta^{-j}.$$

In the case of $\beta$-shift $\Sigma_\beta$, we notice that the measure $\mu_p$ could be considered some what natural construction of a Bernoulli type measure for $\Sigma_\beta$, but in general $\mu_p$ does fail to be, for example, $T_\beta$ invariant under the $\beta$ transformation $T_\beta$. However, what we see that having a type of quasi-Bernoulli is closely related to the algebraic properties of $\beta$:

\begin{theorem}\label{thm:dynamics}
Let $\beta > 1$ and $\Sigma_\beta$ the associated subshift. Then the measure $\mu_p$ is quasi-shift-invariant, that is, the shift action preserve the $\mu_p$ null sets. Moreover, the following are equivalent
\begin{itemize}
\item[(1)] $\beta$ is simple number, that is, the $\beta$-expansion of $1$ is finite;
\item[(2)] $\mu_p$ is quasi-Bernoulli, that is, there is a constant $C > 0$ such that
$$C^{-1}\mu_p[w]\mu_p[v] \leq \mu_p[wv] \leq C\mu_p[w]\mu_p[v]$$
for all admissible $w,v \in \Sigma^*$ with $wv$ admissible.
\item[(3)] $\mu_p$ is strongly quasi-invariant with respect to the shift.
\end{itemize}
When the $\beta$ is simple, by the strong quasi-invariance, there exists a unique ergodic probability measure on $\Sigma_\beta$ equivalent to $\mu_p$.
\end{theorem}

This could be considered as an analogue of Parry's characterisation \cite{Pa60} of subshift of finite type with $\beta$ being simple, and indeed we will use this as an ingredient of the proof.

This work was initiated from the question to establish concrete formulae for the Hausdorff dimensions of the sets of real numbers with specified digit frequencies associated to $\beta$-expansions, and for this purpose Theorem \ref{thm:dynamics} becomes useful. Here we define the level sets
$$F_p:=\Big\{x\in [0, 1): \lim_{n\to\infty}\frac{\sharp\{1\le k\le n: \varepsilon_k(x,\beta)=0\}}{n}=p\Big\},$$
$$\underline{F}_p:=\Big\{x\in [0, 1): \varliminf_{n\to\infty}\frac{\sharp\{1\le k\le n: \varepsilon_k(x,\beta)=0\}}{n}=p\Big\},$$
$$\overline{F}_p:=\Big\{x\in [0, 1): \varlimsup_{n\to\infty}\frac{\sharp\{1\le k\le n: \varepsilon_k(x,\beta)=0\}}{n}=p\Big\}$$
where $\epsilon_1(x,\beta)\epsilon_2(x,\beta)\cdots\epsilon_k(x,\beta)\cdots$ is the $\beta$-expansion of $x$. A well-known result associated to the digit frequencies is the result of Fan and Zhu \cite{FZ04}, who prove that
$$\dim_HF_p=\frac{p\log p-(2p-1)\log(2p-1)-(1-p)\log(1-p)}{\log\beta}$$
where $\beta=\frac{\sqrt{5}+1}{2}$ is the golden ratio and $\frac{1}{2}\le p\le1$.

We employ the random walks on $\Sigma_\beta^*$ above to extend the work \cite{FZ04} to more general numbers and obtain the following extension:

\begin{theorem}\label{thm:digit}
For $1<\beta<2$ such that $\epsilon(1,\beta)=10^m10^\infty$ with some $m\in\{0,1,2,3,\cdots\}$, the following exact formulas of the Hausdorff dimension of $F_p$, $\underline{F}_p$ and $\overline{F}_p$ hold:
\begin{itemize}
\item[(1)] If $0\le p<\frac{m+1}{m+2}$, then $F_p=\underline{F}_p=\overline{F}_p=\emptyset$ and $\dim_H F_p=\dim_H\underline{F}_p=\dim_H\overline{F}_p=0$.
\item[(2)] If $\frac{m+1}{m+2}\le p\le 1$, then $\dim_HF_p=\dim_H\underline{F}_p=\dim_H\overline{F}_p$
\begin{small}
$$=\frac{(mp-m+p)\log(mp-m+p)-(mp-m+2p-1)\log(mp-m+2p-1)-(1-p)\log(1-p)}{\log\beta}.$$
\end{small}
\end{itemize}
\end{theorem}

For calculating the Hausdorff dimension of the level set $F_p$, there is a variation formula in \cite{PfSu05} says that we only need to calculate the measure-theoretic entropy of $T_\beta$ with respect to the invariant probability Borel measure with maximal entropy taking value $p$ on $[0,\frac{1}{\beta})$ (see also \cite[Proposition 4.2]{Li19}). The following two examples show that if we assume that $\beta$ has the form assumed in Theorem \ref{thm:digit}, then $m_p$, the $T_\beta$-ergodic invariant probability Borel measure we study in Section 4, is a measure with maximal entropy:

\begin{example}\label{positive}
Let $\beta\in(1,2)$ such that $\epsilon(1,\beta)=10^m10^\infty$ with some $m\in\{0,1,2,3,\cdots\}$ . Then for any $p\in(0,1)$, we have
\begin{small}
$$h_{m_p}(T_\beta)=\sup\Big\{h_\nu(T_\beta):\nu\text{ is a }T_\beta\text{-invariant }[0,1)\text{ and }\nu[0,\frac{1}{\beta})=m_p[0,\frac{1}{\beta})\Big\}.$$
\end{small}
\end{example}

However, if we do not assume that $\beta$ has the form assumed in Theorem \ref{thm:digit}, then there exists $\beta\in(1,2)$ such that $m_p$ will never be the measure with maximal entropy:

\begin{example}\label{counter}
Let $\beta\in(1,2)$ such that $\epsilon(1,\beta)=1110^\infty$. Then for any $p\in(0,1)$, we have
\begin{small}
$$h_{m_p}(T_\beta)<\sup\Big\{h_\nu(T_\beta):\nu\text{ is a }T_\beta\text{-invariant on }[0,1)\text{ and }\nu[0,\frac{1}{\beta})=m_p[0,\frac{1}{\beta})\Big\}.$$
\end{small}
\end{example}

See Section \ref{sec:examples} for proofs of these examples. As a future problem it would be interesting to see how the random walk we use could be used to characterise further arithmetic properties of $\beta$, and also if one can prove similar results for other $\beta$ transformations like the intermediate $T_{\beta,\alpha}(x) = \beta x + \alpha \mod 1$.

The article is organised as follows. In Section \ref{sec:prelim} we give some notations and preliminaries about the beta-shifts and their properties. In Section \ref{sec:digit} we define the digit frequency parameters and establish some key properties of them using the structure of the beta-shift. In Section \ref{sec:rw} we prove the dynamical properties of the random walk $X_n$ on $\Sigma_\beta^*$. In Sections \ref{sec:locdim} and \ref{sec:hd} we prove local dimension bounds for $\mu_p$ and Hausdorff dimension bounds for the digit frequency sets. Finally, in Section  \ref{sec:examples} we prove the Examples \ref{positive} and \ref{counter}.

\section{Notation and preliminaries}\label{sec:prelim}

Throughout this paper, we use $\N$ to denote the positive integer set $\{1,2,3,4,\cdots\}$ and $\N_{\ge0}$ to denote the non-negative integer set $\{0,1,2,3,\cdots\}$.

In this section, we assume $\beta>1$. We will give some basic notations and recall some necessary preliminary work.

Similar to \cite{DK02}, we consider the \textit{$\beta$-transformation} $T_\beta:[0,1]\rightarrow[0,1)$ given by
$$T_\beta(x):=\beta x-\lfloor\beta x \rfloor\quad\text{for } x\in[0,1]$$
where $\lfloor\beta x \rfloor$ denotes the integer part of $\beta x$. Let
$$\mathcal{A}_\beta:=\left\{\begin{array}{ll}
\{0,1,\cdots,\beta-1\} & \mbox{if } \beta\in\N \\
\{0,1,\cdots,\lfloor\beta\rfloor\} & \mbox{if } \beta\notin\N
\end{array}\right.$$
and for any $n\in\N, x\in[0,1]$, we define
$$\epsilon_n(x,\beta) := \lfloor \beta T^{n-1}_\beta (x)\rfloor\in\cA_\beta.$$
Then we can write
$$x = \sum_{n = 1}^\infty\frac{\epsilon_n(x,\beta)}{\beta^n}$$
and call the sequence $\epsilon(x,\beta):=\epsilon_1(x,\beta)\epsilon_2(x,\beta)\cdots\epsilon_n(x,\beta)\cdots$ the $\beta$-\textit{expansion} of $x$.

We use $\epsilon_1\epsilon_2\cdots\epsilon_n\cdots$ to denote $\epsilon(1,\beta)=\epsilon_1(1,\beta)\epsilon_2(1,\beta)\cdots\epsilon_n(1,\beta)\cdots$ for abbreviation in this paper. We say that $\epsilon(1,\beta)$ is infinite if there are infinitely many $n\in\N$ such that $\epsilon_n\neq 0$. Conversely, if there exists $M\in \mathbb{N}$ such that $j>M$ implies $\epsilon_j=0$, we say that $\epsilon(1,\beta)$ is finite and call $\beta$ a simple beta-number. If additionally $\epsilon_M\neq0$, we say that $\epsilon(1,\beta)$ is finite with length $M$.

The \textit{modified} $\beta$-expansion of $1$ is very useful for showing the admissibility of a sequence (see for example Lemma \ref{charADM}). It is defined by
$$\epsilon^*(1,\beta):=\left\{\begin{array}{ll}
\epsilon(1,\beta) & \mbox{if } \epsilon(1,\beta) \mbox{ is infinite};\\
(\epsilon_1\cdots\epsilon_{M-1}(\epsilon_M-1))^\infty & \mbox{if } \epsilon(1,\beta) \mbox{ is finite with length } M.
\end{array}\right.$$
No matter whether $\epsilon(1,\beta)$ is finite or not, we denote $\epsilon^*(1,\beta)=\epsilon_1^*(1,\beta)\epsilon_2^*(1,\beta)\cdots\epsilon_n^*(1,\beta)\cdots$ by $\epsilon_1^*\epsilon_2^*\cdots\epsilon_n^*\cdots$ for abbreviation.

For a finite word $w$, we use $|w|$  to denote its \textit{length}. On the other hand, we write $w|_k:=w_1w_2\cdots w_k$ to be the prefix of $w$ with length $k$ for $w\in \mathcal{A}_\beta^\N$ or $w\in \mathcal{A}_\beta^n$ where $n\ge k$.

Let $\sigma:\mathcal{A}_\beta^\N\rightarrow\mathcal{A}_\beta^\N$ be the \textit{shift}
$$\sigma(w_1w_2\cdots) = w_2w_3\cdots \quad \text{for } w \in \mathcal{A}_\beta^\N.$$
We define the usual metric $d$ on $\mathcal{A}_\beta^\N$ by
$$d(w,v):=\beta^{-\inf\{k\ge0: w_{k+1}\neq v_{k+1}\}}\quad\text{for any }w,v\in\mathcal{A}_\beta^\N.$$
Then $\sigma$ is continuous.

\begin{definition}[Admissibility] A sequence $w \in \mathcal{A}_\beta^\N$ is called \textit{admissible} if there exists $x \in [0,1)$ such that $\epsilon_i(x,\beta) = w_i$ for all $i \in \N$. We denote the set of all admissible sequences by $\Sigma_\beta$. A word $w \in \mathcal{A}_\beta^n$ is called \textit{admissible} if there exists $x \in [0,1)$ such that $\epsilon_i(x,\beta) = w_i$ for $i = 1,\cdots,n$. We denote the set of all admissible words with length $n$ by $\Sigma_\beta^n$ and write
$$\Sigma_\beta^*:=\bigcup_{n=1}^\infty\Sigma_\beta^n.$$
\end{definition}

\begin{remark}
It is not difficult to check $w|_n\in\Sigma_\beta^n$ and $w_{n+1}w_{n+2}\cdots\in\Sigma_\beta$ for any $n\in\mathbb{N}$ and $w\in\Sigma_\beta$ by definition.
\end{remark}

\begin{lemma}[Parry's criterion \cite{Pa60}]\label{charADM}
Let $w \in \mathcal{A}_\beta^\N$.
Then $w$ is admissible (that is, $w \in \Sigma_\beta$) if and only if
$$\sigma^k (w) \prec \epsilon^*(1,\beta) \quad \text{for all } k \ge 0$$
where $\prec$ means the lexicographic order smaller in $\cA_\beta^\N$.
\end{lemma}

Noting that $\sigma_\beta(\Sigma_\beta)=\Sigma_\beta$, we use $\sigma_\beta:\Sigma_\beta\to\Sigma_\beta$ to denote the restriction of $\sigma$ on $\Sigma_\beta$ and then $(\Sigma_\beta,\sigma_\beta)$ is a dynamical system.

The continuous projection map $\pi_\beta:\Sigma_\beta\rightarrow[0,1)$ defined by
$$\pi_\beta(w)=\frac{w_1}{\beta}+\frac{w_2}{\beta^2}+\cdots+\frac{w_n}{\beta^n}+\cdots \quad \text{for } w \in \Sigma_\beta$$
is bijective with $\epsilon(\cdot,\beta):[0,1)\to\Sigma_\beta$ as its inverse.

\begin{definition}[Cylinder]\label{Cylinder}
Let $w\in\Sigma_\beta^*$. We call
$$[w]:=\{v\in\Sigma_\beta:v_i=w_i \text{ for all }1\le i\le |w|\}$$
the \textit{cylinder} in $\Sigma_\beta$ generated by $w$ and
$$I(w):=\pi_\beta([w])$$
the \textit{cylinder} in $[0,1)$ generated by $w$.
For any $x\in[0,1)$, the cylinder of order $n$ containing $x$ is denoted by
$$I_n(x):=I(\epsilon_1(x,\beta)\epsilon_2(x,\beta)\cdots\epsilon_n(x,\beta)).$$
\end{definition}

\begin{definition}[Full words and cylinders]\label{Deffull}
Let $w\in\Sigma_\beta^n$. If $T_\beta^nI(w)=[0, 1)$, we call the word $w$ and the cylinders $[w],I(w)$ \textit{full}.
\end{definition}

\begin{lemma}[\cite{BuWa14, FW12, LiWu08}]\label{admfull}
Let $w_1\cdots w_n\in\Sigma_\beta^*$ with $w_n \neq 0$. Then  for any $0\le w_n' < w_n$, $w_1 \cdots w_{n-1} w_n'$ is full.
\end{lemma}

\begin{proposition}[\cite{LiLi18}]\label{charFULL}
Let $w \in \Sigma_\beta^n$. Then the following are equivalent.
\begin{itemize}
\item[\emph{(1)}] The word $w$ is full, i.e., $T_\beta^nI(w)=[0, 1)$.
\item[\emph{(2)}] $|I(w)| = \beta^{-n}$.
\item[\emph{(3)}] The sequence $ww'$ is admissible for any $w' \in \Sigma_\beta$.
\item[\emph{(4)}] The word $ww'$ is admissible for any $w' \in \Sigma_\beta^*$.
\item[\emph{(5)}] The word $w\epsilon_1^*\cdots\epsilon_k^*$ is admissible for any $k\ge 1$.
\item[\emph{(6)}] $\sigma^n[w]=\Sigma_\beta$.
\end{itemize}
\end{proposition}

\begin{proposition}[\cite{LiLi18}]\label{fullfullfull}
Let $w, w'\in\Sigma_\beta^*$ be full and $|w|=n\in\N$. Then
\begin{itemize}
\item[\emph{(1)}] the word $ww'$ is full (see also \cite{BuWa14});
\item[\emph{(2)}] the word $\sigma^k (w):=w_{k+1}\cdots w_n$ is full for any $1\le k<n$ ;
\item[\emph{(3)}] the digit $w_n<\lfloor\beta\rfloor$ if $\beta\notin\N$. In particular, $w_n=0$ if $1<\beta<2$.
\end{itemize}
\end{proposition}

\begin{proposition}[\cite{LiLi18}]\label{exaDIV}
(1) Any truncation of $\epsilon(1,\beta)$ is not full (if it is admissible). That is, $\epsilon(1,\beta)|_k$ is not full for any $k \in \N$ (if it is admissible).
\item(2) Let $k \in \mathbb{N}$. Then $\epsilon^*(1,\beta)|_k$ is full if and only if $\epsilon(1,\beta)$ is finite with length $M$ which exactly divides $k$, i.e., $M|k$.
\end{proposition}

\begin{proposition}[\cite{LiLi18}]\label{tail-non-full}Let $w\in\Sigma_\beta^n$. Then $w$ is not full if and only if it ends with a prefix of $\epsilon(1,\beta)$. That is, when $\epsilon(1,\beta)$ is infinite (finite with length $M$), there exists $1\le s\le n$ ( $1\le s\le \min\{M-1,n\}$ respectively) such that $w=w_1\cdots w_{n-s}\epsilon_1\cdots\epsilon_s$.
\end{proposition}

For $n\in\N$, we use $l_n(\beta)$ to denote the number of $0$s following $\epsilon_n^*(1,\beta)$ as in \cite{LiWu08}, i.e.,
$$l_n(\beta):=\sup\{k\ge1:\epsilon_{n+j}^*(1,\beta)=0 \text{ for all } 1\le j\le k\}$$
where by convention $\sup\emptyset:=0$. The set of $\beta>1$ such that the length of the strings of $0$s in $\epsilon^*(1,\beta)$ is bounded is denoted by
$$A_0:=\{\beta>1: \{l_n(\beta)\}_{n\ge1}\text{ is bounded}\}.$$

\begin{proposition}[\cite{LiWu08}]\label{cylinderlength}
Let $\beta>1$. Then $\beta\in A_0$ if and only if there exists a constant $c>0$ such that for all $x\in[0,1)$ and $n\ge1$,
$$c\cdot\frac{1}{\beta^n}\le|I_n(x)|\le\frac{1}{\beta^n}$$
\end{proposition}

\begin{proposition}[\cite{BuWa14} Covering properties]\label{coveringproperties}
Let $\beta>1$. For any $x\in[0,1)$ and any positive integer $n$, the ball $B(x,\beta^{-n})$ intersected with $[0,1)$ can be covered by at most $4(n+1)$ cylinders of order $n$.
\end{proposition}

\begin{definition}[Absolute continuity and equivalence]
Let $\mu$ and $\nu$ be measures on a measurable space $(X,\cF)$. We say that $\mu$ is \textit{absolutely continuous} with respect to $\nu$ and denote it by $\mu \ll \nu$ if $\nu(A) = 0$ implies $\mu(A) = 0$ for any $A \in\cF$. Moreover, if $\mu\ll\nu$ and $\nu\ll\mu$ we say that $\mu$ and $\nu$ are \textit{equivalent} and denote it by $\mu\sim\nu$.
\end{definition}

By the structure of cylinders, the following lemma follows from a similar proof of Lemma 1. (i) in \cite{Wal78}.

\begin{lemma}\label{countabledisjointunionoffull}
Any cylinder (in $\Sigma_\beta$ or $[0,1)$) can be written as a countable disjoint union of full cylinders.
\end{lemma}

In order to extend some properties from a small family to a larger one in some proofs in Section 4, we recall the following two well-known theorems as basic knowledge of measure theory. For more details, see for examples \cite{Co80} and \cite{Dud02}.

\begin{theorem}[Monotone class theorem]\label{monotone class theorem}
Let $\cA$ be an algebra and $M(\cA)$ be the smallest monotone class containing $\cA$. Then $M(\cA)$ is precisely the $\sigma$-algebra generated by $\cA$, i.e., $\sigma(\cA)=M(\cA)$.
\end{theorem}

\begin{theorem}[Dynkin's $\pi$-$\lambda$ theorem]\label{dynkin theorem}
Let $\cC$ be a $\pi$-system and $\cG$ be a $\lambda$-system with $\cC\subset\cG$. Then the $\sigma$-algebra generated by $\cC$ is contained in $\cG$, i.e., $\sigma(\cC)\subset\cG$.
\end{theorem}

The following approximation lemma follows from Theorem 0.1 and Theorem 0.7  in \cite{Wal82}.
\begin{lemma}\label{approximation}
Let $(X,\cB,\mu)$ be a probability space, $\cC$ be a semi-algebra which generates the $\sigma$-algebra $\cB$ and $\cA$ be the algebra generated by $\cC$. Then
\begin{itemize}
\item[(1)] $\cA=\cC_{\Sigma f}:=\{\bigcup_{i=1}^n C_i: C_1,\cdots,C_n\in\cC \text{ are disjoint, } n\in\N\}$;
\item[(2)] for each $\epsilon>0$ and each $B\in\cB$, there is some $A\in\cA$ with $\mu(A\triangle B)<\epsilon$.
\end{itemize}
\end{lemma}

\section{Digit frequency parameters}\label{sec:digit}

Let $1<\beta\le2$. Write
$$\cN_0(w) := \{k\ge 0 :w_{k+1} = 0 \text{ and } w_1w_2\dots w_k 1 \text{ is admissible}\}\quad \text{ for any } w\in\Sigma_\beta,$$
$$\cN_0(w) := \{0 \le k < |w| :w_{k+1} = 0 \text{ and } w_1w_2\dots w_k 1 \text{ is admissible}\}\quad \text{ for any } w\in\Sigma_\beta^*,$$
$$\cN_1(w) := \{k\ge1 : w_{k} = 1  \}\quad\text{ for any } w\in\Sigma_\beta,$$
$$\cN_1(w) := \{1 \le k \le |w| : w_{k} = 1\}\quad\text{ for any } w\in\Sigma_\beta^*$$
and let
$$N_0(w) := \sharp \cN_0(w),\quad N_1(w) := \sharp \cN_1(w)\quad\text{ for any } w\in\Sigma_\beta^*\text{ or }\Sigma_\beta,$$
$$N_0(x,n):=N_0(\epsilon(x,\beta)|_n),\quad N_1(x,n):=N_1(\epsilon(x,\beta)|_n)\quad \text{ for any } x\in[0,1)$$
where $\sharp\cN$ means the cardinality of the set $\cN$.

\begin{remark}Noting that $N_1(w)$ is just the number of the digit $1$ appearing in $w$, it is immediate from the definition that if $w,w' \in \Sigma_\beta^*$ such that $ww' \in \Sigma_\beta^*$, then
$$N_1(ww')=N_1(w)+N_1(w').$$
\end{remark}

Denote the first position where $w$ and $\epsilon^*(1,\beta)$ are different by
$$\mathfrak{m}(w):=\min\{k\ge1:w_k<\epsilon_k^*\} \quad\text{ for } w\in\Sigma_\beta$$
$$\text{and}\quad\mathfrak{m}(w):=\mathfrak{m}(w0^\infty) \quad\text{ for } w\in\Sigma_\beta^*.$$
For any $w\in\Sigma_\beta$, combing the facts $w\prec\epsilon^*(1,\beta)$, $\epsilon^*(1,\beta)|_n\in\Sigma_\beta^*,\forall n\in\N$ and Lemma \ref{admfull}, we know that there exists $k\in\N$ such that $w|_k$ is full. Therefore we can write
$$\tau(w):=\min\{k\ge1:w|_k \text{ is full}\} \quad\text{ for any } w\in\Sigma_\beta,$$
$$\text{and}\quad\tau(w):=\tau(w0^\infty) \quad\text{ for any } w\in\Sigma_\beta^*.$$
For any $w\in\Sigma_\beta^*$, regarding $w|_0$ as the empty word which is full, we write
$$\tau'(w):=\max\{0\le k\le|w|:w|_k\text{ is full}\}.$$

\begin{lemma}\label{tau}
Let $\beta>1$. For any $w\in\Sigma_\beta\cup\Sigma_\beta^*$, we have
$$\tau(w)=\left\{\begin{array}{ll}
\mathfrak{m}(w) & \mbox{if } \epsilon(1,\beta) \mbox{ is infinite};\\
\min\{\mathfrak{m}(w),M\} & \mbox{if } \epsilon(1,\beta)\mbox{ is finite with length }M.
\end{array}\right.$$
\end{lemma}
\begin{proof}
For any $w\in\Sigma_\beta\cup\Sigma_\beta^*$. Let $k=\mathfrak{m}(w).$ Then $w|_k=\epsilon^*_1\cdots\epsilon^*_{k-1}w_k$ and $w_k<\epsilon^*_k$. (When $w\in\Sigma_\beta^*$ and $k>|w|$, we regard $w|_k=w_1\cdots w_k$ as $w_1\cdots w_{|w|}0^{k-|w|}$). By $\epsilon^*_1\cdots\epsilon^*_{k-1}\epsilon^*_k\in\Sigma^*_\beta$ and Lemma \ref{admfull}, $w|_k$ is full.
\item(1) When $\epsilon(1,\beta)$ is infinite, for any $1\le i\le k-1$, we have $w|_i=\epsilon^*(1,\beta)|_i=\epsilon(1,\beta)|_i$ which is not full by Proposition \ref{exaDIV}. Therefore $\tau(w)=k=\mathfrak{m}(w)$.
\item(2) when $\epsilon(1,\beta)=\epsilon_1\cdots\epsilon_M0^\infty$ with $\epsilon_M\neq0$:
\newline \textcircled{\footnotesize{$1$}} If $k\le M$, then for any $1\le i\le k-1<M$, we have $w|_i=\epsilon^*(1,\beta)|_i$ which is not full by Proposition \ref{exaDIV}. Therefore $\tau(w)=k=\mathfrak{m}(w)$.
\newline \textcircled{\footnotesize{$2$}} If $k>M$, then $w|_M=\epsilon^*(1,\beta)|_M$ is full by Proposition \ref{exaDIV}. For any $1\le i\le M-1$, we have $w|_i=\epsilon^*(1,\beta)|_i$ which is not full by Proposition \ref{exaDIV}. Therefore $\tau(w)=M$.
\end{proof}

\begin{lemma}\label{infinitefullwords} Let $\beta>1$ and $w\in\Sigma_\beta$. Then
\begin{itemize}
\item[\emph{(1)}] there exists a strictly increasing sequence $(n_j)_{j \geq 1} $ such that $w|_{n_j}$ is full for any $j \in \N$;
\item[\emph{(2)}] $N_0(w)=+\infty$ if $1<\beta\le2$.
\end{itemize}
\end{lemma}
\begin{proof}
\item(1) Let $k_1:=\mathfrak{m}(w)$, $n_1:=k_1$, $k_j:=\mathfrak{m}(\sigma^{n_{j-1}}w)$ and $n_j:=n_{j-1}+k_j$ for any $j\ge2$. Then $n_j$ is strictly increasing. By $\epsilon^*_1\cdots\epsilon^*_{k_1-1}\epsilon^*_{k_1}\in\Sigma^*_\beta$, $w_{n_1}<\epsilon^*_{k_1}$ and Lemma \ref{admfull}, we know that $w_1\cdots w_{n_1-1}w_{n_1}=\epsilon^*_1\cdots\epsilon^*_{k_1-1}w_{n_1}$ is full. Similarly for any $j\ge2$, by $\epsilon^*_1\cdots\epsilon^*_{k_j-1}\epsilon^*_{k_j}\in\Sigma^*_\beta$, $w_{n_j}<\epsilon^*_{k_j}$ and Lemma \ref{admfull}, we know that $w_{n_{j-1}+1}\cdots w_{n_j-1}w_{n_j}=\epsilon^*_1\cdots\epsilon^*_{k_j-1}w_{n_j}$ is full. Therefore, by Proposition \ref{fullfullfull} (1), $w|_{n_j}$ is full for any $j\in\N$.
\item(2) Noting that $1<\beta\le2$, by $w_{n_j}<\epsilon^*_{k_j}$, we get $w_{n_j}=0, \epsilon^*_{k_j}=1$ for any $j\in\N$. Thus
    $$w_1\cdots w_{n_j-1}1=\epsilon^*_1\cdots\epsilon^*_{k_1-1}w_{n_1}\cdots\cdots\epsilon^*_1\cdots\epsilon^*_{k_{j-1}-1}w_{n_{j-1}}\epsilon^*_1\cdots\epsilon^*_{k_j-1}\epsilon^*_{k_j}\in\Sigma^*_\beta$$
    for any $j\in\N$ by Proposition \ref{fullfullfull} (1) and Proposition \ref{charFULL} (5). Therefore $N_0(w)=+\infty$.
\end{proof}

\begin{lemma}\label{parameter}
Let $1<\beta\le2$, $w,w' \in \Sigma_\beta^*$ with $ww' \in \Sigma_\beta^*$. Then
\begin{itemize}
\item[\emph{(1)}] $N_0(w) \le N_0(ww') \le N_0(w) + N_0(w')$;
\item[\emph{(2)}] when $w$ is full, we have $N_0(ww') = N_0(w) + N_0(w')$;
\item[\emph{(3)}] when $\epsilon(1,\beta)=\epsilon_1\cdots\epsilon_M0^{\infty}$ with $\epsilon_M\neq 0$, we have $N_0(ww') \ge N_0(w) + N_0(w')-M$.
\end{itemize}
\end{lemma}
\begin{proof}
Let $a=|w|, b=|w'|$ and then $ww'=w_1\cdots w_a w'_1\cdots w'_b$.
\item(1) \textcircled{\footnotesize{$1$}} $N_0(w)\le N_0(ww')$ follows from $\cN_0(w)\subset\cN_0(ww')$.
\newline\textcircled{\footnotesize{$2$}} Prove $N_0(ww') \le N_0(w) + N_0(w')$.
\begin{itemize}
\item[i)] We prove $\cN_0(ww')\subset\cN_0(w)\cup(\cN_0(w')+a)$ first. Let $k\in\cN_0(ww')$.
  \newline If $0\le k<a$, then $w_{k+1}=0$ and $w_1\cdots w_k 1\in\Sigma_\beta^*$. We get $k\in\cN_0(w)$.
  \newline If $a\le k<a+b$, then $w'_{k-a+1}=0$ and $w_1\cdots w_a w'_1\cdots w'_{k-a} 1\in\Sigma_\beta^*$. It follows from $w'_1\cdots w'_{k-a}1\in\Sigma_\beta^*$ that $k-a\in\cN_0(w')$ and $k\in\cN_0(w')+a$.
\item[ii)] Combining $\cN_0(w)\cap(\cN_0(w')+a)=\emptyset$, $\sharp(\cN_0(w')+a)=\sharp\cN_0(w')$ and i), we get $N_0(ww')\le N_0(w)+N_0(w')$.
\end{itemize}
\item(2) We need to prove $N_0(ww')\ge N_0(w)+N_0(w')$. By $\sharp\cN_0(w')=\sharp(\cN_0(w')+a)$, it suffices to prove $\cN_0(ww')\supset\cN_0(w)\cup(\cN_0(w')+a)$. For each $k\in\cN_0(w)$, obviously $k\in\cN_0(ww')$. On the other hand, if $k\in(\cN_0(w')+a)$, then $k-a\in\cN_0(w')$, $w'_{k-a+1}=0$ and $w'_1\cdots w'_{k-a}1\in\Sigma_\beta^*$. Since $w$ is full, by Proposition \ref{charFULL}, we get $ww'_1\cdots w'_{k-a}1\in\Sigma_\beta^*$ and then $k\in\cN_0(ww')$.
\item(3) \textcircled{\footnotesize{$1$}} Firstly, we divide $ww'$ into three segments.
\begin{itemize}
\item[i)] Let $k_0:=\tau'(w)$, then $0\le k_0\le a$. If $k_0=a$, $w$ is full. Then the conclusion follows from (2) immediately. Therefore we assumes $0\le k_0<a$ in the following proof. Let $u^{(1)}:=w_1\cdots w_{k_0}$ be full and $|u^{(1)}|=k_0$. (When $k_0=0$, we regard $u^{(1)}$ as the empty word and $N_0(u^{(1)}):=0$.)
\item[ii)] Consider $w_{k_0+1}\cdots w_a w'_1\cdots w'_b\in\Sigma_\beta^*$ (the admissibility follows from $ww'\in\Sigma_\beta^*$). \newline Let $k_1:=\tau(w_{k_0+1}\cdots w_a w'_1\cdots w'_b)\ge1$. By the definition of $k_0=\tau'(w)$ and Proposition \ref{fullfullfull}, we get $k_1>a-k_0$. In the following, we assume $k_1\le a-k_0+b$ first. The case $k_1>a-k_0+b$ will be considered at the end of the proof. Let $u^{(2)}:=w_{k_0+1}\cdots w_a w'_1\cdots w'_{k_0+k_1-a}$, then $|u^{(2)}|=k_1$.
\item[iii)] Let $u^{(3)}:=w'_{k_0+k_1-a+1}\cdots w'_b$. (When $k_0+k_1-a=b$, we regard $u^{(3)}$ as the empty word and $N_0(u^{(3)}):=0$.)
\end{itemize}
\indent Up to now, we write $ww'=u^{(1)}u^{(2)}u^{(3)}$.
$$\underbrace{w_1\cdots w_{k_0}}_{|u^{(1)}|=k_0}\underbrace{w_{k_0+1}\cdots w_a w'_1 \cdots w'_{k_0+k_1-a}}_{|u^{(2)}|=k_1}\underbrace{w'_{k_0+k_1-a+1}\cdots w'_b}_{|u^{(3)}|}$$
\newline\textcircled{\footnotesize{$2$}} Estimate $N_0(ww'), N_0(w)$ and $N_0(w')$.
\begin{itemize}
\item[i)] $N_0(ww')=N_0(u^{(1)}u^{(2)}u^{(3)})\xlongequal[\text{by (2)}]{u^{(1)}\text{ full}}N_0(u^{(1)})+N_0(u^{(2)}u^{(3)})\xlongequal[\text{by (2)}]{u^{(2)}\text{ full}}N_0(u^{(1)})+N_0(u^{(2)})+N_0(u^{(3)})$.
\item[ii)] $N_0(w)\xlongequal[\text{by (2)}]{u^{(1)}\text{ full}}N_0(u^{(1)})+N_0(w_{k_0+1}\cdots w_a)\overset{\text{by (1)}}{\le}N_0(u^{(1)})+N_0(u^{(2)})$.
\item[iii)] $N_0(w')\overset{\text{by (1)}}{\le}N_0(w'_1\cdots w'_{k_0+k_1-a})+N_0(u^{(3)})\le M+N_0(u^{(3)})$ where the last inequality follows from
    $$N_0(w'_1\cdots w'_{k_0+k_1-a})\le k_0+k_1-a\le k_1=\tau(w_{k_0+1}\cdots w_a w'_1\cdots w'_b)\overset{\text{by Lemma \ref{tau}}}{\le} M$$
\end{itemize}
Combining i), ii) and iii), we get $N_0(ww')\ge N_0(w)+N_0(w')-M$.
\newline
\newline\indent To end the proof, it suffices to consider the case $k_1>a-k_0+b$ below. We define $u^{(1)}$ as before and define $u^{(2)}:=w_{k_0+1}\cdots w_a w'_1\cdots w'_b$ which is not full. Then $|u^{(2)}|=a-k_0+b$. We do not define $u^{(3)}$.
\newline\textcircled{\footnotesize{$1$}} Prove $N_0(u^{(2)})=0$.
\newline By contradiction, we suppose $N_0(u^{(2)})\neq0$, then there exists $k\in\cN_0(u^{(2)})$, $0\le k<a-k_0+b$ such that $u^{(2)}_{k+1}=0$ and $u^{(2)}_1\cdots u^{(2)}_k 1\in\Sigma_\beta^*$. By Lemma \ref{admfull}, $u^{(2)}_1\cdots u^{(2)}_{k+1}$ is full which contradict $\tau(u^{(2)})=k_1>a-k_0+b$.
\newline\textcircled{\footnotesize{$2$}} Estimate $N_0(ww'), N_0(w)$ and $N_0(w')$.
\begin{itemize}
\item[i)] $N_0(ww')=N_0(u^{(1)}u^{(2)})\xlongequal[\text{by (2)}]{u^{(1)}\text{ full}}N_0(u^{(1)})+N_0(u^{(2)})\xlongequal[]{\text{by }\textcircled{\footnotesize{$1$}}}N_0(u^{(1)})$.
\item[ii)] $N_0(w)\xlongequal[\text{by (2)}]{u^{(1)}\text{ full}}N_0(u^{(1)})+N_0(w_{k_0+1}\cdots w_a)=N_0(u^{(1)})$ where the last equality follows from $N_0(w_{k_0+1}\cdots w_a)\le N_0(u^{(2)})=0$.
\item[iii)] $N_0(w')\le b\le|u^{(2)}|=a-k_0+b<k_1=\tau(u^{(2)})\overset{\text{by Lemma \ref{tau}}}{\le} M$.
\end{itemize}
Combining i), ii) and iii), we get $N_0(ww')\ge N_0(w)+N_0(w')-M$.
\end{proof}

\section{Dynamical properties of the random walk on $\Sigma_\beta^*$}\label{sec:rw}

Recall that the random walk $(X_n)$ in $\Sigma_\beta^*$ defines a probability distribution $\mu_p$ supported on the subshift $\Sigma_\beta$ by setting
$$\mu_p[w] := \P(X_{|w|} = w)$$
for all $w \in \Sigma^*$ and cylinder $[w]$, which then satisfies
$$\mu_p[0] = p, \quad \mu_p[1] = 1-p,$$
\and if $w1 \notin \cW$, we have
$$\mu_p[w0] = p\mu_p[w] \quad \text{and} \quad \mu_p[w1] = (1-p)\mu_p[w].$$
If $w1 \in \cW$, we have
$$\mu_p[w0] = \mu_p[w].$$
Then $\mu_p$ defines a natural probability measure $\nu_p = \pi_\beta \mu_p$ on $[0,1]$ under the natural projection
$$\pi(w) = \sum_{j=1}^\infty w_j \beta^{-j}.$$

\begin{remark}\label{measureformula} (1) By the definition of $\mu_p$ and $\nu_p$, we have
\begin{eqnarray*}
\nu_p(I(w)) = \mu_p[w] &=& p^{N_0(w)} (1-p)^{N_1(w)}\quad \text{for any } w \in \Sigma^*_\beta;\\
\nu_p(I(w|_n)) = \mu_p[w|_n] &=& p^{N_0(w|_n)} (1-p)^{N_1(w|_n)}\quad \text{for any } w \in \Sigma_\beta,n\in\N;\\
\nu_p(I_n(x))=\mu_p[\epsilon(x,\beta)|_n]&=&p^{N_0(x,n)}(1-p)^{N_1(x,n)}\quad\text{for any } x\in[0,1),n\in\N.
\end{eqnarray*}
(2) For any $w\in\Sigma_\beta$, as $n\to+\infty$, by Lemma \ref{infinitefullwords} (2) we get $N_0(w|_n)\to+\infty$ and then $\mu_p[w|_n]\to0$.
\end{remark}

\begin{proposition}\label{noatom}
The measures $\mu_p$, $\sigma_\beta^k\mu_p$, $\nu_p$ and $T_\beta^k\nu_p$ have no atoms. That is, $\mu_p(\{w\})=\sigma_\beta^k\mu_p(\{w\})=\nu_p(\{x\})=T_\beta^k\nu_p(\{x\})=0$ for any single point $w\in \Sigma_\beta$, $x\in[0,1)$ and $k\in\N$.
\end{proposition}
\begin{proof} It follows immediately from $\mu_p[w|_n]\to0$, $\sharp\sigma_\beta^{-k}\{w\}\le2^k$, $\sharp\pi_\beta^{-1}\{x\}=1$ and $\sharp T_\beta^{-k}\{x\}\le2^k$ for any $w\in \Sigma_\beta$ and $x\in[0,1)$.
\end{proof}

\begin{definition}[Invariance and ergodicity]
Let $(X, \mathcal{F}, \mu, T)$ be a measure-preserving dynamical system, that is, $(X,\cF,\mu)$ is a probability space and $\mu$ is $T$-\textit{invariant}, i.e., $T\mu = \mu$. We say that the probability measure $\mu$ is \textit{ergodic} with respect to $T$ if for every $A\in\cF$ satisfying $T^{-1}A=A$ (such a set is called $T$-\textit{invariant}), we have $\mu(A)=0$ or $1$. We also say that $(X,\cF,\mu,T)$ is ergodic.
\end{definition}

Note that $\mu_p$ is not $\sigma_\beta$-invariant and $\nu_p$ is not $T_\beta$-invariant. For example, if $\beta=\frac{1+\sqrt{5}}{2}$ is the golden ratio, then we have
$$\Sigma_\beta^*=\{w\in\bigcup_{n=1}^\infty\{0,1\}^n: 11\text{ does not appear in } w\}.$$
Hence
$$\mu_p [1] =1-p, \ \ \mbox{but} \ \ \mu_p(\sigma_\beta^{-1}[1])=\mu_p [01] =p(1-p).$$
Correspondingly,
$$\nu_p[\frac{1}{\beta},1)=1-p, \ \ \mbox{but} \ \ \nu_p(T_\beta^{-1}[\frac{1}{\beta},1))=p(1-p).$$
We recall the notion of quasi-invariance.

\begin{definition}[Quasi-invariance]\label{quasi-invariant}
Let $(X,\cF,\mu)$ be a measure space and $T$ be a measurable transformation on it. Then
\begin{itemize}
\item[(1)] $\mu$ is \textit{quasi-invariant} with respect to the transformation $T$ if $\mu$ and its image measure $T\mu$ are mutually absolutely continuous (i.e. equivalent), that is,
$$\mu \ll T\mu \ll \mu \quad(\text{i.e. }T\mu\sim\mu);$$
\item[(2)] $\mu$ is \textit{strongly quasi-invariant} with respect to the transformation $T$ if there exists a constant $C  > 0$ such that
$$C^{-1}\mu(A) \leq T^k\mu(A) \leq C\mu(A)$$
for any $k \in \N$ and $A\in\cF$. We also say $\mu$ is \textit{$C$-strongly quasi-invariant} if we know such a $C$.
\end{itemize}
\end{definition}

\begin{definition}[Quasi-Bernoulli]
A measure $\mu$ on $(\Sigma_\beta,\cB(\Sigma_\beta))$ is called \textit{quasi-Bernoulli} if there exists a constant $C > 0$ such that
$$C^{-1}\mu[w]\mu[w'] \leq \mu[ww'] \leq C\mu[w]\mu[w']$$
for every pair $w,w' \in \Sigma_\beta^*$ satisfying $ww' \in \Sigma_\beta^*$.
\end{definition}

\begin{theorem}\label{theoremquasiinvariance}
Let $1<\beta\le 2$ and $0 < p < 1$. Then \begin{itemize}
\item[\emph{(1)}] $\mu_p$ is quasi-invariant with respect to $\sigma_\beta$;
\item[\emph{(2)}] $\epsilon(1,\beta)$ is finite if and only if $\mu_p$ is quasi-Bernoulli;
\item[\emph{(3)}] $\epsilon(1,\beta)$ is finite if and only if $\mu_p$ is strongly quasi-invariant with respect to $\sigma_\beta$.
\end{itemize}
\end{theorem}

The proof of this is based on the following lemma.

\begin{lemma}\label{quasibernoulli}
Let $1<\beta\le2$, $0<p<1$ and $w,w'\in\Sigma_\beta^*$ with $ww'\in\Sigma_\beta^*$. Then
\begin{itemize}
\item[\emph{(1)}] $$\mu_p[w] \geq \mu_p[ww'] \geq \mu_p[w]\mu_p[w'];$$
\item[\emph{(2)}] when $w$ is full, we have
$$\mu_p[ww']=\mu_p[w]\mu_p[w'];$$
\item[\emph{(3)}] if additionally $\epsilon(1,\beta)=\epsilon_1\cdots\epsilon_M0^{\infty}$ with $\epsilon_M\neq 0$, then
$$\mu_p[ww'] \leq p^{-M}\mu_p[w]\mu_p[w'].$$
In particular, $\mu_p$ is quasi-Bernoulli.
\end{itemize}
\end{lemma}
\begin{proof} It follows from Remark \ref{measureformula}, Lemma \ref{parameter} and $N_1(ww')=N_1(w)+N_1(w')$ for any $ww'\in\Sigma_\beta^*$.
\end{proof}

\begin{proof}[Proof of Theorem \ref{theoremquasiinvariance}]
\item(1) \textcircled{\footnotesize{$1$}} Prove $\mu_p\ll\sigma_\beta\mu_p$.
\newline Let $A\in\cB(\Sigma_\beta)$ with $\sigma_\beta\mu_p(A)=0$. It suffices to prove $\mu_p(A)=0$. For any $\epsilon>0$, by
$$\mu_p(\sigma_\beta^{-1}A)=\inf\{\sum_n\mu_p[w^{(n)}]: w^{(n)}\in\Sigma_\beta^*, \sigma_\beta^{-1}A\subset\bigcup_n[w^{(n)}]\}=0,$$
there exists $\{w^{(n)}\}\subset\Sigma_\beta^*$ such that
$$\sigma_\beta^{-1}A\subset\bigcup_n[w^{(n)}]\text{ and }\sum_n\mu_p[w^{(n)}]<\epsilon.$$
Since $\epsilon$ can be small enough such that $\mu_p[0]=p$ and $\mu_p[1]=1-p>\epsilon$, we can assume $a_n:=|w^{(n)}|\ge2$ for any $n$ without loss of generality. By the fact that $\sigma_\beta$ is surjective, we get
$$A=\sigma_\beta(\sigma_\beta^{-1}A)\subset\sigma_\beta(\bigcup_n[w^{(n)}])\subset\bigcup_n\sigma_\beta[w^{(n)}]=\bigcup_n\sigma_\beta[w^{(n)}_1w^{(n)}_2\cdots w^{(n)}_{a_n}]\subset\bigcup_n[w^{(n)}_2\cdots w^{(n)}_{a_n}].$$
Therefore
\begin{eqnarray*}
\mu_p(A)&\le&\sum_n\mu_p[w^{(n)}_2\cdots w^{(n)}_{a_n}]\\
&\le&\frac{1}{\min\{p,1-p\}}\sum_n\mu_p[w^{(n)}_1]\mu_p[w^{(n)}_2\cdots w^{(n)}_{a_n}]\\
&\le&\frac{1}{\min\{p,1-p\}}\sum_n\mu_p[w^{(n)}]\\
&<&\frac{\epsilon}{\min\{p,1-p\}}
\end{eqnarray*}
for any $\epsilon>0$.
\newline\textcircled{\footnotesize{$2$}} Prove $\sigma_\beta\mu_p\ll\mu_p$.
\newline Let $B\in\cB(\Sigma_\beta)$ with $\mu_p(B)=0$. It suffices to prove $\sigma_\beta\mu_p(B)=0$. For any $m\in\N_{\ge2}$, we define $B_m:=B\setminus[\epsilon^*_2\cdots\epsilon^*_m]$.
\begin{itemize}
\item[i)] Prove that $\sigma_\beta\mu_p(B_m)$ increase to $\sigma_\beta\mu_p(B)$.
  \newline\textcircled{\footnotesize{$a$}} If $\epsilon(1,\beta)$ is finite, then $\epsilon^*_2\epsilon^*_3\epsilon^*_4\cdots\notin\Sigma_\beta$, $[\epsilon^*_2\cdots\epsilon^*_m]$ decrease to $\emptyset$, $B_m$ increase to $B$ and $\sigma_\beta\mu_p(B_m)$ increase to $\sigma_\beta\mu_p(B)$.
  \newline\textcircled{\footnotesize{$b$}} If $\epsilon(1,\beta)$ is infinite, then $\epsilon^*_2\epsilon^*_3\epsilon^*_4\cdots=\epsilon_2\epsilon_3\epsilon_4\cdots=\epsilon(T_\beta1,\beta)\in\Sigma_\beta$, $[\epsilon^*_2\cdots\epsilon^*_m]$ decrease to $\{\epsilon^*_2\epsilon^*_3\epsilon^*_4\cdots\}$ (a single point set), $B_m$ increase to $(B\setminus\{\epsilon^*_2\epsilon^*_3\epsilon^*_4\cdots\})$ and $\sigma_\beta\mu_p(B_m)$ increase to $\sigma_\beta\mu_p(B\setminus\{\epsilon^*_2\epsilon^*_3\epsilon^*_4\cdots\})$. Since $\sigma_\beta\mu_p$ has no atom (by Proposition \ref{noatom}), we get $\sigma_\beta\mu_p(B_m)$ increase to $\sigma_\beta\mu_p(B)$.
\item[ii)] In order to get $\sigma_\beta\mu_p(B)=0$, by i) it suffices to prove that for any $m\in\N_{\ge2}$, $\sigma_\beta\mu_p(B_m)=0$.
  \newline Fix $m\in\N_{\ge2}$. By $\mu_p(B_m)\le\mu_p(B)=0$, we get
  $$\inf\Big\{\sum_n\mu_p[w^{(n)}]:w^{(n)}\in\Sigma_\beta^*, B_m\subset\bigcup_n[w^{(n)}]\Big\}=0.$$
  For any $\epsilon>0$, there exists $\{w^{(n)}\}_{n\in N'}\subset\Sigma_\beta^*$ with
  $$B_m\subset\bigcup_{n\in N'}[w^{(n)}] \quad\text{such that}\quad \sum_{n\in N'}\mu_p[w^{(n)}]<\epsilon$$
  where $N'$ is an index set with cardinality at most countable. Since $\epsilon$ can be small enough such that
  $$\delta_m:=\min\{\mu_p[w]:w\in\Sigma_\beta^*, |w|\le m-1\}>\epsilon,$$
  we can assume $a_n:=|w^{(n)}|\ge m$ for all $n\in N'$. Let
  $$N:=\{n\in N':w^{(n)}|_{m-1}\neq\epsilon^*_2\cdots\epsilon^*_m\}\subset N'.$$
  By the fact that for any $n\in N$, $[w^{(n)}]\cap[\epsilon^*_2\cdots\epsilon^*_m]=\emptyset$ and for any $n\in N'\setminus N$, $[w^{(n)}]\subset[\epsilon^*_2\cdots\epsilon^*_m]$, we get
  \begin{eqnarray*}
  B_m&=&B_m\setminus[\epsilon^*_2\cdots\epsilon^*_m]\subset\bigcup_{n\in N'}\big([w^{(n)}]\setminus[\epsilon^*_2\cdots\epsilon^*_m]\big)\\
  &=&\big(\bigcup_{n\in N}\big([w^{(n)}]\setminus[\epsilon^*_2\cdots\epsilon^*_m]\big)\big)\bigcup\big(\bigcup_{n\in N'\setminus N}\big([w^{(n)}]\setminus[\epsilon^*_2\cdots\epsilon^*_m]\big)\big)=\bigcup_{n\in N}[w^{(n)}]
  \end{eqnarray*}
  and then $\sigma_\beta^{-1}B_m\subset\bigcup_{n\in N}\sigma_\beta^{-1}[w^{(n)}]$. Let
  $$N_0:=\{n\in N: 1w^{(n)}\notin\Sigma_\beta^*\}\text{ and }N_1:=\{n\in N:1w^{(n)}\in\Sigma_\beta^*\}.$$
  Then for any $n\in N_0$, $\sigma_\beta^{-1}[w^{(n)}]=[0w^{(n)}]$ and for any $n\in N_1$, $\sigma_\beta^{-1}[w^{(n)}]=[0w^{(n)}]\cup[1w^{(n)}]$. Thus
  $$\sigma_\beta^{-1}B_m\subset\big(\bigcup_{n\in N}[0w^{(n)}]\big)\bigcup\big(\bigcup_{n\in N_1}[1w^{(n)}]\big)$$
  and
  $$\mu_p(\sigma_\beta^{-1}B_m)\le\sum_{n\in N}\mu_p[0w^{(n)}]+\sum_{n\in N_1}\mu_p[1w^{(n)}]=:J_1+J_2$$
  where by Lemma \ref{quasibernoulli} (2),
  $$J_1=\sum_{n\in N}p\mu_p[w^{(n)}]\le p\sum_{n\in N'}\mu_p[w^{(n)}]<p\epsilon.$$
  Now we estimate the upper bounded of $T_2$.
  \newline For each $n\in N_1\subset N$, by $1w^{(n)}_1\cdots w^{(n)}_{m-1}\neq\epsilon^*_1\epsilon^*_2\cdots\epsilon^*_m$, there exists $1\le k_n\le m-1$ such that $1=\epsilon^*_1, w^{(n)}_1=\epsilon^*_2, \cdots w^{(n)}_{k_n-1}=\epsilon^*_{k_n}$ and $w^{(n)}_{k_n}<\epsilon^*_{k_n+1}$. Since $\epsilon^*_1\cdots\epsilon^*_{k_n}\epsilon^*_{k_n+1}\in\Sigma_\beta^*$, by Lemma \ref{admfull} and Proposition \ref{fullfullfull} (2), we know that both $1w^{(n)}_1\cdots w^{(n)}_{k_n}$ and $w^{(n)}_1\cdots w^{(n)}_{k_n}$ are full. It follows from Lemma \ref{quasibernoulli} (2) that
  $$\mu_p[1w^{(n)}]=\mu_p[1w^{(n)}_1\cdots w^{(n)}_{k_n}]\mu_p[w^{(n)}_{k_n+1}\cdots w^{(n)}_{a_n}]$$
  and
  $$\mu_p[w^{(n)}]=\mu_p[w^{(n)}_1\cdots w^{(n)}_{k_n}]\mu_p[w^{(n)}_{k_n+1}\cdots w^{(n)}_{a_n}].$$
  Let
  $$C_m:=\max\Big\{\frac{\mu_p[1w]}{\mu_p[w]}:w\in\Sigma_\beta^*\text{ with }1w\in\Sigma_\beta^*\text{ and }1\le|w|\le m-1\Big\}<\infty.$$
  By $k_n\le m-1$, we get $\mu_p[1w^{(n)}]\le C_m\mu_p[w^{(n)}]$ for any $n\in N_1$. This implies
  $$J_2=\sum_{n\in N_1}\mu_p[1w^{(n)}]\le C_m\sum_{n\in N_1}\mu_p[w^{(n)}]\le C_m\sum_{n\in N'}\mu_p[w^{(n)}]<C_m\epsilon.$$
  Therefore $\mu_p(\sigma_\beta^{-1}B_m)<(p+C_m)\epsilon$ for any $0<\epsilon<\delta_m$. We conclude that $\sigma_\beta\mu_p(B_m)=0$.
\end{itemize}

\item(2) $\boxed{\Rightarrow}$ follows from Lemma \ref{quasibernoulli}.
\newline$\boxed{\Leftarrow}$ (By contradiction) Assume that $\epsilon(1,\beta)=\epsilon_1\epsilon_2\epsilon_3\cdots$ is infinite. By $\epsilon_2\epsilon_3\cdots=\epsilon(T_\beta1,\beta)\in\Sigma_\beta$ and Lemma \ref{infinitefullwords} (2), we get $N_0(\epsilon_2\epsilon_3\cdots)=+\infty$. Then for any $N\in\N$, there exists $n\in\N$ such that $N_0(\epsilon_2\epsilon_3\cdots\epsilon_n)\ge N$. Let $w:=\epsilon_1=1$ and $w':=\epsilon_2\epsilon_3\cdots\epsilon_n$. Then $ww'=\epsilon_1\cdots\epsilon_n$ and obviously
$$N_0(ww')=0=0+N-N\le N_0(w)+N_0(w')-N.$$
By Remark \ref{measureformula} (1) and $N_1(ww')=N_1(w)+N_1(w')$, we get
$$\mu_p[ww']\ge p^{-N}\mu_p[w]\mu_p[w'].$$
Since for any $N\in\N$, there exists $w,w'$ which satisfy the above inequality and $p^{-N}$ can be arbitrary large, we know that $\mu_p$ is not quasi-Bernoulli.

\item(3) $\boxed{\Leftarrow}$ (By contradiction) Assume that $\epsilon(1,\beta)=\epsilon_1\epsilon_2\epsilon_3\cdots$ is infinite. By $\epsilon_2\epsilon_3\cdots=\epsilon(T_\beta1,\beta)\in\Sigma_\beta$ and Lemma \ref{infinitefullwords} (2), we get $N_0(\epsilon_2\epsilon_3\cdots)=+\infty$. Then for any $N\in\N$, there exists $n\in\N$ such that $N_0(\epsilon_2\epsilon_3\cdots\epsilon_n)\ge N$. Let $w:=\epsilon_2\cdots\epsilon_n$. Then
$$\sigma_\beta\mu_p[w]=\mu_p[0w]+\mu_p[1w]\ge\mu_p[\epsilon_1\epsilon_2\cdots\epsilon_n]=p^{N_0(\epsilon_1\cdots\epsilon_n)}(1-p)^{N_1(\epsilon_1\cdots\epsilon_n)}=(1-p)^{N_1(\epsilon_1\cdots\epsilon_n)}$$
and
$$\mu_p[w]=p^{N_0(w)}(1-p)^{N_1(w)}\le p^N(1-p)^{N_1(\epsilon_1\cdots\epsilon_n)}.$$
Thus
$$\sigma_\beta\mu_p[w]\ge(1-p)p^{-N}\mu_p[w].$$
Since for any $N\in\N$, there exists $w$ which satisfy the above inequality and $(1-p)p^{-N}$ can be arbitrary large, we know that $\mu_p$ is not strongly quasi-invariant.
\newline$\boxed{\Rightarrow}$ Let $\epsilon(1,\beta)=\epsilon_1\cdots\epsilon_M0^\infty$ with $\epsilon_M\neq0$ and $c=p^{-M}>0$.
\newline \textcircled{\footnotesize{$1$}} Prove $c^{-1}\mu_p[w]\le\sigma_\beta^k\mu_p[w]\le c\mu_p[w]$ for all $k\in\N$ and $w\in\Sigma_\beta^*$.
\newline Notice that
$$\sigma_\beta^{-k}[w]=\bigcup_{u_1\cdots u_kw\in\Sigma_\beta^*}[u_1\cdots u_kw]$$
is a disjoint union.
\begin{itemize}
\item[i)]Estimate the upper bound of $\sigma_\beta^k\mu_p[w]$:
\begin{eqnarray*}
\mu_p\sigma_\beta^{-k}[w]&=&\sum_{u_1\cdots u_kw\in\Sigma_\beta^*}\mu_p[u_1\cdots u_kw]\\
&\overset{\textcircled{\footnotesize{$a$}}}{\le}&\sum_{u_1\cdots u_kw\in\Sigma_\beta^*}p^{-M}\mu_p[u_1\cdots u_k]\mu_p[w]\\
&\le& p^{-M}\sum_{u_1\cdots u_k\in\Sigma_\beta^*}\mu_p[u_1\cdots u_k]\mu_p[w]\\
&=&p^{-M}\mu_p[w].
\end{eqnarray*}
where \textcircled{\footnotesize{$a$}} follows from Lemma \ref{quasibernoulli}.
\item[ii)]Estimate the lower bound of $\sigma_\beta^k\mu_p[w]$:
$$\mu_p\sigma_\beta^{-k}[w]=\sum_{u_1\cdots u_kw\Sigma_\beta^*}\mu_p[u_1\cdots u_kw]\ge\sum_{u_1\cdots u_{k-M}0^M\Sigma_\beta^*}\mu_p[u_1\cdots_{k-M}0^Mw].$$
(Without loss of generality, we assume $k\ge M$. Otherwise, we consider $0^kw$ instead of $u_1\cdots u_{k-M}0^Mw$). By Proposition \ref{tail-non-full}, $u_1\cdots u_{k-m}0^M$ is full for any $u_1\cdots u_{k-m}\in\Sigma_\beta^*$. Then by Proposition \ref{charFULL} (4), we get
$$u_1\cdots u_{k-M}0^Mw\in\Sigma_\beta^*\Longleftrightarrow u_1\cdots u_{k-M}\in\Sigma_\beta^*.$$
Therefore
\begin{eqnarray*}
\mu_p\sigma_\beta^{-k}[w]&\ge&\sum_{u_1\cdots u_{k-M}\in\Sigma_\beta^*}\mu_p[u_1\cdots u_{k-M}0^Mw]\\
&\overset{\textcircled{\footnotesize{$b$}}}{=}&\sum_{u_1\cdots u_{k-M}\in\Sigma_\beta^*}\mu_p[u_1\cdots u_{k-M}0^M]\mu_p[w]\\
&\overset{\textcircled{\footnotesize{$c$}}}{\ge}&\sum_{u_1\cdots u_{k-M}\in\Sigma_\beta^*}\mu_p[u_1\cdots u_{k-M}]p^M\mu_p[w]\\
&=&p^M\mu_p[w]
\end{eqnarray*}
where \textcircled{\footnotesize{$b$}} and \textcircled{\footnotesize{$c$}} follow from Lemma \ref{quasibernoulli} (2) and (1) respectively.
\end{itemize}
\textcircled{\footnotesize{$2$}} Prove $c^{-1}\mu_p(B)\le\sigma_\beta^k\mu_p(B)\le c\mu_p(B)$ for all $k\in\N$ and $B\in\cB(\Sigma_\beta)$.
\newline Let $\cC:=\{[w]:w\in\Sigma_\beta^*\}\cup\{\emptyset\}$, $\cC_{\Sigma f}:=\{\bigcup_{i=1}^n C_i:C_1,\cdots,C_n\in\cC \text{ are disjoint, } n\in\N\}$ and
$$\cG:=\{B\in\cB(\Sigma_\beta):c^{-1}\mu_p(B)\le\sigma_\beta^k\mu_p(B)\le c\mu_p(B) \text{ for all }k\in\N\}.$$
Then $\cC$ is a semi-algebra, $\cC_{\Sigma f}$ is the algebra generated by $\cC$ (by Theorem \ref{approximation} (1)) and $\cG$ is a monotone class. Since in \textcircled{\footnotesize{$1$}} we have already $\cC\subset\cG$, it is obvious that $\cC_{\Sigma f}\subset\cG\subset\cB(\Sigma_\beta)$. By Monotone Class Theorem (Theorem \ref{monotone class theorem}), we get $\cG=\cB(\Sigma_\beta)$.
\end{proof}

By Theorem \ref{theoremquasiinvariance}, we get the following.

\begin{corollary}\label{nustrong} Let $1<\beta\le2$ and $0<p<1$. Then
\begin{itemize}
\item[(1)] $\nu_p$ is quasi-invariant with respect to $T_\beta$;
\item[(2)] $\epsilon(1,\beta)$ is finite if and only if $\nu_p$ is strongly quasi-invariant with respect to $T_\beta$.
\end{itemize}
\end{corollary}

\begin{theorem}\label{mp}
Let $1<\beta\le2$ and $0<p<1$. If $\epsilon(1,\beta)$ is finite, then there exists a unique $T_\beta$-ergodic probability measure $m_p$ on $([0,1),\cB[0,1))$ equivalent to $\nu_p$, where $m_p$ is defined by
$$m_p(B):=\lim_{n\to\infty}\frac{1}{n}\sum_{k=0}^{n-1}T_\beta^k\nu_p(B)\quad\text{for }B\in\cB[0,1).$$
\end{theorem}

The proof of this is based on the following lemmas.

\begin{lemma}[\cite{DM46}]\label{aeexist}
Let $(X,\cB,\mu)$ be a probability space and $T$ be a measurable transformation on $X$ satisfying $\mu(T^{-1}E)=0$ whenever $E\in\cB$ with $\mu(E)=0$. If there exists a constant $M$ such that for any $E\in\cB$ and any $n\ge1$,
$$\frac{1}{n}\sum_{k=0}^{n-1}\mu(T^{-k}E)\le M\mu(E),$$
then for any real integrable function $f$ on $X$, the limit
$$\lim_{n\to\infty}\frac{1}{n}\sum_{k=0}^{n-1}f(T^kx)$$
exists for $\mu$-almost every $x\in X$.
\end{lemma}

\begin{lemma}\label{01}
Let $1<\beta\le2$ and $0<p<1$.
\begin{itemize}
\item[(1)] If $B\in\cB(\Sigma_\beta)$ with $\sigma_\beta^{-1}B=B$, then $\mu_p(B)=0$ or $1$.
\item[(2)] If $B\in\cB[0,1)$ with $T_\beta^{-1}B=B$, then $\nu_p(B)=0$ or $1$.
\end{itemize}
\end{lemma}
\begin{proof}
\item(1) Let $\cF:=\{w\in\Sigma_\beta^*:w\text{ is full}\}$.
\begin{itemize}
\item[\textcircled{\footnotesize{$1$}}] Let $w\in\cF$ with $|w|=n$. We prove $\mu_p([w]\cap\sigma_\beta^{-n}A)=\mu_p[w]\mu_p(A)$ for any $A\in\cB(\Sigma_\beta)$ as below.
    \newline Since $w$ is full and $[ww']=[w]\cap\sigma_\beta^{-n}[w']$ for any $w'\in\Sigma_\beta^*$, we get
    $$\mu_p([w]\cap\sigma_\beta^{-n}[w'])=\mu_p[ww']\xlongequal[]{\text{by Lemma \ref{quasibernoulli} (2)}}\mu_p[w]\mu_p[w'].$$
    Let $\cC:=\{[w']:w'\in\Sigma_\beta^*\}\cup\{\emptyset\}$ and $\cG:=\{A\in\cB(\Sigma_\beta):\mu_p([w]\cap\sigma_\beta^{-n}A)=\mu_p[w]\mu_p(A)\}$. Then we have already got $\cC\subset\cG\subset\cB(\Sigma_\beta)$. Since $\cC$ is a $\pi$-system, $\cG$ is a $\lambda$-system and $\cC$ generates $\cB(\Sigma_\beta)$, by Dynkin's $\pi$-$\lambda$ Theorem \ref{dynkin theorem}, we get $\cG=\cB(\Sigma_\beta)$.
\item[\textcircled{\footnotesize{$2$}}] We use $B^c$ to denote the complement of $B$ in $\Sigma_\beta$. For any $\delta>0$, by Lemma \ref{approximation} and Lemma \ref{countabledisjointunionoffull}, there exists a countable disjoint union of full cylinders $E_\delta=\bigcup_i[w^{(i)}]$ with $\{w^{(i)}\}\subset\cF$ such that $\mu_p(B^c\triangle E_\delta)<\delta$.
\item[\textcircled{\footnotesize{$3$}}] Let $B\in\cB(\Sigma_\beta)$ with $\sigma_\beta^{-1}B=B$. Then $B=\sigma_\beta^{-n}B$ and by \textcircled{\footnotesize{$1$}} we get
    $$\mu_p(B\cap[w])=\mu_p(\sigma_\beta^{-n}B\cap[w])=\mu_p(B)\mu_p[w]$$
    for any $w\in\cF$. Thus
    $$\mu_p(B\cap E_\delta)=\mu_p(B\cap\bigcup_i[w^{(i)}])=\sum_i\mu_p(B\cap[w^{(i)}])=\sum_i\mu_p(B)\mu_p[w^{(i)}]=\mu_p(B)\mu_p(E_\delta).$$
    Let $a=\mu_p((B\cup E_\delta)^c)$, $b=\mu_p(B\cap E_\delta)$, $c=\mu_p(B\setminus E_\delta)$ and $d=\mu_p(E_\delta\setminus B)$. Then
    $$b=(b+c)(b+d),\quad a+b<\delta\text{  (by \textcircled{\footnotesize{$2$}})}\quad\text{and}\quad a+b+c+d=1.$$
    By
    $$(b+c)(a+d-\delta)\le(b+c)(b+d)=b<\delta,$$
    we get
    $$(b+c)(a+d)<(1+b+c)\delta\le2\delta$$
    which implies $\mu_p(B)\mu_p(B^c)\le2\delta$ for any $\delta>0$. Therefore $\mu_p(B)=0$ or $\mu_p(B^c)=0$.
\end{itemize}
\item(2) follows from (1). In fact, let $B\in\cB[0,1)$ with $T_\beta^{-1}B=B$. By $\sigma_\beta^{-1}\pi_\beta^{-1}B=\pi_\beta^{-1}T_\beta^{-1}B=\pi_\beta^{-1}B\in\cB(\Sigma_\beta)$ and (1), we get $\mu_p(\pi_\beta^{-1}B)=0$ or $1$, i.e., $\nu_p(B)=0$ or $1$.
\end{proof}

\begin{proof}[Proof of Theorem \ref{mp}]
\item(1) For any $n\in\N$ and $B\in\cB[0,1)$, define
$$m_p^n(B):=\frac{1}{n}\sum_{k=0}^{n-1}\nu_p(T_\beta^{-k}B).$$
Then $m_p^n$ is a probability measure on $([0,1),\cB[0,1))$. By Corollary \ref{nustrong}, there exists $c>0$ such that
\begin{align}\label{equi}c^{-1}\nu_p(B)\le m_p^n(B)\le c\nu_p(B)\quad\text{for any }B\in\cB[0,1)\text{ and  }n\in\N.\end{align}
\item(2) For any $B\in\cB[0,1)$, prove that $\lim_{n\to\infty}m_p^n(B)$ exists. In fact,
\begin{eqnarray*}
\lim_{n\to\infty}m_p^n(B)=\lim_{n\to\infty}\frac{1}{n}\sum_{k=0}^{n-1}\int\mathbbm{1}_{T_\beta^{-k}B}d\nu_p&=&\lim_{n\to\infty}\int\frac{1}{n}\sum_{k=0}^{n-1}\mathbbm{1}_B(T_\beta^kx)d\nu_p(x)\\
&=&\int\lim_{n\to\infty}\frac{1}{n}\sum_{k=0}^{n-1}\mathbbm{1}_B(T_\beta^kx)d\nu_p(x).
\end{eqnarray*}

The last equality follows from Dominate Convergence Theorem where the $\nu_p$-a.e. existence of $\lim\limits_{n\to\infty}\frac{1}{n}\sum\limits_{k=0}^{n-1}\mathbbm{1}_B(T_\beta^kx)$ follows from Lemma \ref{aeexist}, the strongly quasi-invariance of $\nu_p$ and (1).
\item(3) For any $B\in\cB[0,1)$, define $m_p(B):=\lim_{n\to\infty}m_p^n(B)$. Then $m_p$ is a probability measure on $([0,1),\cB[0,1))$.
\item(4) $m_p\sim\nu_p$ on $\cB[0,1)$ follows from \eqref{equi} and the definition of $m_p$.
\item(5) Prove that $m_p$ is $T_\beta$-invariant.
\newline For any $B\in\cB[0,1)$ and $n\in\N$, we have
    $$m_p^n(T_\beta^{-1}B)=\frac{1}{n}\sum_{k=1}^n\nu_p(T_\beta^{-k}B)=\frac{n+1}{n}m_p^{n+1}(B)-\frac{\nu_p(B)}{n}.$$
    As $n\to\infty$, we get $m_p(T_\beta^{-1}B)=m_p(B)$.
\item(6) Prove that $([0,1),\cB[0,1),m_p,T_\beta)$ is ergodic.
\newline Let $B\in\cB[0,1)$ such that $T_\beta^{-1}B=B$. Then by Lemma \ref{01} (2), we get $\nu_p(B)=0$ or $\nu_p(B^c)=0$ which implies $m_p(B)=0$ or $m_p(B^c)=0$ since $m_p\sim\nu_p$. Noting that $m_p$ is $T_\beta$-invariant, we know that $m_p$ is ergodic with respect to $T_\beta$.
\item(7) Prove that such $m_p$ is unique on $\cB[0,1)$.
\newline Let $m_p'$ be a $T_\beta$-ergodic probability measure on $([0,1),\cB[0,1))$ equivalent to $\nu_p$. Then for any $B\in\cB[0,1)$, by the Birkhoff Ergodic Theorem, we get
    $$m_p(B)=\int\mathbbm{1}_Bdm_p=\lim_{n\to\infty}\frac{1}{n}\sum_{k=0}^{n-1}\mathbbm{1}_B(T_\beta^kx)\quad\text{for }m_p\text{-a.e. }x\in[0,1)$$
    and
    $$m_p'(B)=\int\mathbbm{1}_Bdm_p'=\lim_{n\to\infty}\frac{1}{n}\sum_{k=0}^{n-1}\mathbbm{1}_B(T_\beta^kx)\quad\text{for }m_p'\text{-a.e. }x\in[0,1).$$
    Since $m_p\sim\nu_p\sim m_p'$, there exists $x\in[0,1)$ such that $m_p(B)=\lim_{n\to\infty}\frac{1}{n}\sum_{k=0}^{n-1}\mathbbm{1}_B(T_\beta^k x)=m_p'(B)$.
\end{proof}

\section{Modified lower local dimension related to $\beta$-expansions}\label{sec:locdim}

Let $\nu$ be a finite measure on $\R^n$. The \textit{lower local dimension} of $\nu$ at $x\in\R^n$ is defined by
$$\underline{\dim}_{loc}\nu(x):=\varliminf_{r\to 0}\frac{\log\nu(B(x,r))}{\log r},$$
where $B(x,r)$ is the closed ball centered on $x$ with radius $r$. Theoretically, we can use the lower local dimension to estimate the upper and lower bounds of the Hausdorff dimension (see \cite{Fal90} for definition) by the following proposition.
\begin{proposition}[\cite{Fal97} Proposition 2.3]\label{oringin}
Let $s\ge0$, $E\subset\R^n$ be a Borel set and $\nu$ be a finite Borel measure on $\R^n$.
\begin{itemize}
\item[\emph{(1)}] If $\underline{\dim}_{loc}\nu(x)\le s$ for all $x\in E$ then $\dim_HE\le s$.
\item[\emph{(2)}] If $\underline{\dim}_{loc}\nu(x)\ge s$ for all $x\in E$ and $\nu(E)>0$ then $\dim_HE\ge s$.
\end{itemize}
\end{proposition}

But in the definition of the lower local dimension, the Bernoulli-type measure of a ball $\nu_p(B(x,r))$ is difficult to estimate. Therefore, we use the measure of a cylinder $\nu(I_n(x))$ instead of $\nu_p(B(x,r))$ to define the \textit{modified lower local dimension related to $\beta$-expansions} of a measure at a point.

\begin{definition}
Let $\beta>1$ and $\nu$ be a finite measure on $[0,1)$. The \textit{modified lower local dimension} of $\nu$ at $x\in[0,1)$ is defined by
$$\underline{\dim}_{loc}^\beta\nu(x):=\varliminf_{n\to\infty}\frac{\log\nu(I_n(x))}{\log|I_n(x)|}$$
where $I_n(x)$ is the cylinder of order $n$ containing $x$.
\end{definition}

Combining Proposition \ref{oringin} (1) and the following proposition, we can estimate the upper bound of the Hausdorff dimension by the modified lower local dimension.

\begin{proposition}\label{relation}
Let $\beta>1$ and $\nu$ be a finite measure on $[0,1)$. Then for any $x\in[0,1)$, $$\underline{\dim}_{loc}^\beta(\nu,x)\ge\underline{\dim}_{loc}(\nu,x).$$
\end{proposition}
\begin{proof}
For any $x\in[0,1)$ and $n\in\N$. Let $r_n:=|I_n(x)|$, then $I_n(x)\subset B(x,r_n)$, $\nu(I_n(x))\le\nu(B(x,r_n))$ and $-\log\nu(I_n(x))\ge-\log\nu(B(x,r_n))$. We get $$\frac{-\log\nu(I_n(x))}{-\log|I_n(x)|}\ge\frac{-\log\nu(B(x,r_n))}{-\log r_n}.$$
Therefore $$\varliminf_{n\to\infty}\frac{\log\nu(I_n(x))}{\log|I_n(x)|}\ge\varliminf_{n\to\infty}\frac{\log\nu(B(x,r_n))}{\log r_n}\ge\underline{\dim}_{loc}\nu(x).$$
\end{proof}

\begin{remark}\label{wrongexample}
The reverse inequality in Proposition \ref{relation}, i.e., $\underline{\dim}_{loc}^\beta(\nu,x)\le\underline{\dim}_{loc}(\nu,x)$ is not always true. For example, let $\beta$ be the golden ratio $(\sqrt{5}+1)/2$, $x=\beta^{-1}$ and $\nu=\nu_p$ be the $(p,1-p)$ Bernoulli-type measure with $0<p<1/2$. For any $n\in\N$, let $r_n=|I_n(x)|$ and $J_n$ be the left consecutive cylinder of $I_n(x)$ with the same  order $n$. When $n\ge2$, we have $r_n=\beta^{-n}\ge|J_n|$ and $B(x,r_n)\supset J_n$. Then $\nu_p(B(x,r_n))\ge\nu_p(J_n)\ge p(1-p)^{n-1}$ and $\nu_p(I_n(x))=(1-p)p^{n-2}$ which implies
$$\underline{\dim}_{loc}^\beta\nu_p(x)=\varliminf_{n\to\infty}\frac{\log(1-p)p^{n-2}}{\log\beta^{-n}}=\frac{-\log p}{\log\beta}$$
and
$$\underline{\dim}_{loc}\nu_p(x)\le\varliminf_{n\to\infty}\frac{\log\nu_p(B(x,r_n))}{\log r_n}\le\varliminf_{n\to\infty}\frac{\log p(1-p)^{n-1}}{\log\beta^{-n}}=\frac{-\log(1-p)}{\log\beta}.$$
When $0<p<1/2$, we have $\underline{\dim}_{loc}^\beta(\nu_p,x)>\underline{\dim}_{loc}(\nu_p,x)$.
\end{remark}

Though the reverse inequality in Proposition \ref{relation} is not always true, we are going to establish the following theorem for estimating both of the upper and lower bounds of the Hausdorff dimension by the modified lower local dimension of a finite measure.

\begin{theorem}\label{dimestimate}
Let $\beta>1$, $s\ge0$, $E\subset[0,1)$ be a Borel set and $\nu$ be a finite Borel measure on $[0,1)$.
\begin{itemize}
\item[\emph{(1)}] If $\underline{\dim}_{loc}^\beta\nu(x)\le s$ for all $x\in E$, then $\dim_HE\le s$.
\item[\emph{(2)}] If $\underline{\dim}_{loc}^\beta\nu(x)\ge s$ for all $x\in E$ and $\nu(E)>0$, then $\dim_HE\ge s$.
\end{itemize}
\end{theorem}
\begin{proof} (1) follows from Proposition \ref{oringin} (1) and Proposition \ref{relation}.
\newline(2) follows from the following Lemma \ref{measurelowerbound}. In fact, if $s=0$, $\dim_H E\ge s$ is obvious. If $s>0$, let $0<t<s$. For any $x\in E$, by $\varliminf\limits_{n\to\infty}\frac{\log\nu(I_n(x))}{\log|I_n(x)|}\ge s>t$, there exists $N\in\N$ such that any $n>N$ implies $\frac{\log\nu(I_n(x))}{\log|I_n(x)|}>t$ and $\nu(I_n(x))<|I_n(x)|^t$. So $\varlimsup\limits_{n\to\infty}\frac{\nu(I_n(x))}{|I_n(x)|^t}\le1<2$. For any $0<\epsilon<t$, by Lemma \ref{measurelowerbound}, we get $\cH^{t-\epsilon}(E)>0$ (where $\cH^s(E)$ denotes the classical $s$-dimension Hausdorff measure of a set $E$.) and then $\dim_H E\ge t-\epsilon$. So $\dim_HE\ge t$ for any $t<s$. Therefore $\dim_HE\ge s$.
\end{proof}

\begin{remark}
The statement (2) in Theorem \ref{dimestimate} obviously implies the Proposition 1.3 in \cite{BuWa14} which is called the modified mass distribution principle.
\end{remark}

\begin{lemma}\label{measurelowerbound}
Let $\beta>1, s>0, c>0$, $E\subset [0,1)$ be a Borel set and $\nu$ be a finite Borel measure on $[0,1)$. If $\varlimsup\limits_{n\to\infty}\frac{\nu(I_n(x))}{|I_n(x)|^s}< c$ for all $ x\in E$, then for any $0<\epsilon<s$, $\cH^{s-\epsilon}(E)\ge c^{-1}\nu(E)$.
\end{lemma}
\begin{proof} It follows immediately from Lemma \ref{netmeasurelowerbound} and Lemma \ref{netmeasure}.
\end{proof}

For establishing this lemma, we need the followings.

Let $\beta>1$, $s\ge0$ and $E\subset [0,1)$. For any $\delta>0$, we define $$\cH^{s,\beta}_\delta(E):=\inf\Big\{\sum_k|J_k|^s:|J_k|\le\delta, E\subset \bigcup_kJ_k, \{J_k\} \text{ are countable cylinders}\Big\}.$$
It is increasing as $\delta\searrow0$. We call $\cH^{s,\beta}(E):=\lim_{\delta\to 0}\cH^{s,\beta}_\delta(E)$ the $s$-dimension Hausdorff measure of $E$ related to the cylinder net of $\beta$.

\begin{lemma}\label{netmeasure}
Let $\beta>1$, $s>0$ and $E\subset [0,1)$. Then for any $0<\epsilon<s$, $\cH^{s,\beta}(E)\le\cH^{s-\epsilon}(E)$.
\end{lemma}
\begin{proof} Fix $0<\epsilon<s$.
\item(1) Choose $\delta_0>0$ as below.
\newline Since $\beta^{(n+1)\epsilon}\to\infty$ much faster than $8\beta^sn\to\infty$ as $n\to\infty$, there exists $n_0\in\N$ such that for any $n>n_0$, $8\beta^s n\le\beta^{(n+1)\epsilon}$. By $\frac{-\log\delta}{\log\beta}-1\to\infty$ as $\delta\to0^+$, there exists $\delta_0>0$ small enough such that $\frac{-\log\delta_0}{\log\beta}-1>n_0$. Then for any $n>\frac{-\log\delta_0}{\log\beta}-1$, we will have $8\beta^sn\le\beta^{(n+1)\epsilon}$.
\item(2) In order to arrive at the conclusion, it suffices to prove for any $0<\delta<\delta_0$, $\cH^{s,\beta}_{\beta\delta}(E)\le\cH_\delta^{s-\epsilon}(E)$.
\newline Fix $0<\delta<\delta_0$. Let $\{U_i\}$ be a $\delta$-cover of $E$, i.e., $0<|U_i|\le\delta$ and $E\subset\cup_iU_i$. Then for each $U_i$, there exists $n_i\in\N$ such that $\beta^{-n_i-1}<|U_i|\le\beta^{-n_i}$. By Proposition \ref{coveringproperties}, $U_i$ can be covered by at most $8n_i$ cylinders $I_{i,1}, I_{i,2}, \cdots, I_{i,8n_i}$ of order $n_i$. Noting that
$$|I_{i,j}|\le\beta^{-n_i}<\beta|U_i|\le\beta\delta\text{ and }E\subset\bigcup_i\bigcup_{j=1}^{8n_i}I_{i,j},$$
we get
$$\cH^{s,\beta}_{\beta\delta}(E)\le\sum_i\sum_{j=1}^{8n_i}|I_{i,j}|^s\le\sum_i\frac{8n_i}{\beta^{n_is}}\overset{(\star)}{\le}\sum_i\frac{1}{\beta^{(n_i+1)(s-\epsilon)}}<\sum_i|U_i|^{s-\epsilon}.$$
Taking $\inf$ on the right, we conclude that $\cH^{s,\beta}_{\beta\delta}(E)\le\cH_\delta^{s-\epsilon}(E)$.
\newline( $(\star)$ is because $\frac{1}{\beta^{n_i+1}}<|U_i|<\delta_0$ implies $n_i>\frac{-\log\delta_0}{\log\beta}-1$ and then $8n_i\beta^s\le\beta^{(n_i+1)\epsilon}$.)
\end{proof}

\begin{lemma}\label{netmeasurelowerbound}
Let $\beta>1, s\ge0, c>0$, $E\subset [0,1)$ be a Borel set and $\nu$ be a finite Borel measure on $[0,1)$. If $\varlimsup_{n\to\infty}\frac{\nu(I_n(x))}{|I_n(x)|^s}< c$ for all $x\in E$, then $\cH^{s,\beta}(E)\ge c^{-1}\nu(E)$.
\end{lemma}
\begin{proof} For any $\delta>0$, let $E_\delta:=\{x\in E: |I_n(x)|<\delta$ implies $\nu(I_n(x))<c|I_n(x)|^s\}$.
\newline(1) Prove that when $\delta\searrow0$, $E_\delta\nearrow E$ as below.
\newline\textcircled{\footnotesize{$1$}} If $0<\delta_2<\delta_1$, then obviously $E_{\delta_2}\supset E_{\delta_1}$.
\newline\textcircled{\footnotesize{$2$}} It suffices to prove $E=\bigcup_{\delta>0}E_\delta$.
\newline$\boxed{\supset}$ follows from $E\supset E_\delta, \forall\delta>0$.
\newline$\boxed{\subset}$ Let $x\in E$. By $\varlimsup\limits_{n\to\infty}\frac{\nu(I_n(x))}{|I_n(x)|^s}<c$, there exists $N_x\in\N$ such that any $n>N_x$ will have $\nu(I_n(x))<c|I_n(x)|^s$. Let $\delta_x=|I_{N_x}(x)|$, then $|I_n(x)|<\delta_x$ will imply $n>N_x$ and $\nu(I_n(x))<c|I_n(x)|^s$. Therefore $x\in E_{\delta_x}\subset\bigcup_{\delta>0}E_\delta$.
\newline(2) Fix $\delta>0$. Let $\{J_k\}_{k\in K}$ be countable cylinders such that $|J_k|<\delta$ and $\bigcup_{k\in K}J_k\supset E\supset E_\delta$. Let $K'=\{k\in K: J_k\cap E_\delta\neq\emptyset\}$. For any $k\in K'$, there exists $x_k\in J_k\cup E_\delta$. By the definition of $E_\delta$, we get $\nu(J_k)<c|J_k|^s$. So
$$\nu(E_\delta)\le\nu(\bigcup_{k\in K'}J_k)\le\sum_{k\in K'}\nu(J_k)<\sum_{k\in K'}c|J_k|^s\le c\sum_{k\in K}|J_k|^s.$$
Taking $\inf$ on the right, we get $\nu(E_\delta)\le c\cH^{s,\beta}_\delta(E)\le c\cH^{s,\beta}(E)$. Let $\delta\to 0$ on the left, by $E_\delta\nearrow E$, we conclude that $\nu(E)\le c\cH^{s,\beta}(E)$.
\end{proof}

\section{Hausdorff dimension of some level sets}\label{sec:hd}

We apply the Bernoulli-type measures and the modified lower local dimension related to $\beta$-expansions to give some new results on the Hausdorff dimension of level sets in this section.

For $1<\beta\le2$ and $0\le p\le 1$, consider the following \textit{level sets}
$$F_p:=\Big\{x\in [0, 1): \lim_{n\to\infty}\frac{\sharp\{1\le k\le n: \varepsilon_k(x,\beta)=0\}}{n}=p\Big\},$$
$$\underline{F}_p:=\Big\{x\in [0, 1): \varliminf_{n\to\infty}\frac{\sharp\{1\le k\le n: \varepsilon_k(x,\beta)=0\}}{n}=p\Big\},$$
$$\overline{F}_p:=\Big\{x\in [0, 1): \varlimsup_{n\to\infty}\frac{\sharp\{1\le k\le n: \varepsilon_k(x,\beta)=0\}}{n}=p\Big\}.$$
Obviously, $F_p=\underline{F}_p\cap\overline{F}_p$.

\begin{theorem}[Upper bound of the Hausdorff dimension of level sets]\label{upbounddimention} Let $1<\beta\le2$ and $0\le p\le 1$. Then
$$\dim_H F_p\le\min\{\dim_H \underline{F}_p,\dim_H \overline{F}_p\}\le\max\{\dim_H \underline{F}_p,\dim_H \overline{F}_p\}\le\frac{-p\log p-(1-p)\log(1-p)}{\log\beta}.$$
In particular, $\dim_HF_0=\dim_H\underline{F}_0=\dim_H\overline{F}_0=\dim_HF_1=\dim_H\underline{F}_1=\dim_H\overline{F}_1=0$.
\end{theorem}
\begin{proof} First, we consider $0<p<1$.
\newline For any $x\in[0,1)$ and $n\in\N$, it follows from $\nu_p(I_n(x))=p^{N_0(x,n)}(1-p)^{N_1(x,n)}$ that
$$-\log\nu_p(I_n(x))=N_0(x,n)(-\log p)+N_1(x,n)(-\log(1-p))$$
$$\le(n-N_1(x,n))(-\log p)+N_1(x,n)(-\log(1-p)).$$
By $|I_n(x)|\le\beta^{-n}$, we get
\begin{align}\label{upbound}\frac{-\log\nu_p(I_n(x))}{-\log|I_n(x)|}\le\frac{(1-\frac{N_1(x,n)}{n})(-\log p)+\frac{N_1(x,n)}{n}(-\log(1-p))}{\log\beta}.
\end{align}
\item(1) For any $x\in\underline{F}_p$, it follows from $\varliminf\limits_{n\to\infty}(1-\frac{N_1(x,n)}{n})=p$ and $\varlimsup\limits_{n\to\infty}\frac{N_1(x,n)}{n}=1-p$ that
\begin{eqnarray*}
\varliminf_{n\to\infty}\frac{\log\nu_p(I_n(x))}{\log|I_n(x)|}&\le&\frac{\varliminf\limits_{n\to\infty}(1-\frac{N_1(x,n)}{n})(-\log p)+\varlimsup\limits_{n\to\infty}\frac{N_1(x,n)}{n}(-\log (1-p))}{\log\beta}\\
&=&\frac{-p\log p-(1-p)\log(1-p)}{\log\beta}.
\end{eqnarray*}
By Theorem \ref{dimestimate} (1), we get
$$\dim_H \underline{F}_p\le\frac{-p\log p-(1-p)\log(1-p)}{\log\beta}.$$
\item(2) For any $x\in\overline{F}_p$, it follows from $\varlimsup\limits_{n\to\infty}(1-\frac{N_1(x,n)}{n})=p$ and $\varliminf_{n\to\infty}\frac{N_1(x,n)}{n}=1-p$ that
\begin{eqnarray*}
\varliminf\limits_{n\to\infty}\frac{\log\nu_p(I_n(x))}{\log|I_n(x)|}&\le&\frac{\varlimsup\limits_{n\to\infty}(1-\frac{N_1(x,n)}{n})(-\log p)+\varliminf\limits_{n\to\infty}\frac{N_1(x,n)}{n}(-\log (1-p))}{\log\beta}\\
&=&\frac{-p\log p-(1-p)\log(1-p)}{\log\beta}.
\end{eqnarray*}
By Theorem \ref{dimestimate} (1), we get
$$\dim_H \overline{F}_p\le\frac{-p\log p-(1-p)\log(1-p)}{\log\beta}.$$
Therefore, by $F_p=\underline{F}_p\cap\overline{F}_p$, we get
$$\dim_H F_p\le\min\{\dim_H \underline{F}_p,\dim_H \overline{F}_p\}\le\max\{\dim_H \underline{F}_p,\dim_H \overline{F}_p\}\le\frac{-p\log p-(1-p)\log(1-p)}{\log\beta}.$$

Before proving $\dim_HF_0=\dim_H\underline{F}_0=\dim_H\overline{F}_0=\dim_HF_1=\dim_H\underline{F}_1=\dim_H\overline{F}_1=0$, we establish the following.

\begin{lemma}\label{varlevelsets}
Let $1<\beta\le2$ and $0<p<1$.
\begin{itemize}
\item[(1)] Let
$$\underline{F}_{\le p}:=\Big\{x\in [0,1): \varliminf_{n\to\infty}\frac{\sharp\{1\le k\le n: \varepsilon_k(x,\beta)=0\}}{n}\le p\Big\}.$$
Then
$$\dim_H \underline{F}_{\le p}\le\frac{-p\log p-\log(1-p)}{\log\beta}.$$
\item[(2)] Let
$$\overline{F}_{\ge p}:=\Big\{x\in [0,1): \varlimsup_{n\to\infty}\frac{\sharp\{1\le k\le n: \varepsilon_k(x,\beta)=0\}}{n}\ge p\Big\}.$$
Then
$$\dim_H \overline{F}_{\ge p}\le\frac{-\log p-(1-p)\log(1-p)}{\log\beta}.$$
\end{itemize}
\end{lemma}
\begin{proof}
\begin{itemize}
\item[(1)] For any $x\in \underline{F}_{\le p}$, it follows from \eqref{upbound}, $\varliminf\limits_{n\to\infty}(1-\frac{N_1(x,n)}{n})\le p$ and $\frac{N_1(x,n)}{n}\le1$ $(\forall n\in\N)$ that
    $$\varliminf_{n\to\infty}\frac{\log\nu_p(I_n(x))}{\log|I_n(x)|}\le\frac{-p\log p-\log(1-p)}{\log\beta}.$$
    By Theorem \ref{dimestimate} (1), we get
    $$\dim_H \underline{F}_{\le p}\le\frac{-p\log p-\log(1-p)}{\log\beta}.$$
\item[(2)] For any $x\in \overline{F}_{\ge p}$, it follows from \eqref{upbound}, $\varliminf\limits_{n\to\infty}\frac{N_1(x,n)}{n}\le 1-p$ and $1-\frac{N_1(x,n)}{n}\le1$ $(\forall n\in\N)$ that
    $$\varliminf_{n\to\infty}\frac{\log\nu_p(I_n(x))}{\log|I_n(x)|}\le\frac{-\log p-(1-p)\log(1-p)}{\log\beta}.$$
    By Theorem \ref{dimestimate} (1), we get
    $$\dim_H \overline{F}_{\ge p}\le\frac{-\log p-(1-p)\log(1-p)}{\log\beta}.$$
\end{itemize}
\end{proof}

Now we prove $\dim_HF_0=\dim_H\underline{F}_0=\dim_H\overline{F}_0=\dim_HF_1=\dim_H\underline{F}_1=\dim_H\overline{F}_1=0$.
\item(1) For any $0<p<1$, $F_0=\overline{F}_0\subset\underline{F}_0\subset \underline{F}_{\le p}$ implies $\dim_HF_0=\dim_H\overline{F}_0\le\dim_H\underline{F}_0\le \dim_H\underline{F}_{\le p}$. Let $p\to 0$, by Lemma \ref{varlevelsets} (1), we get $\dim_HF_0=\dim_H\overline{F}_0=\dim_H\underline{F}_0=0$.
\item(2) For any $0<p<1$, $F_1=\underline{F}_1\subset\overline{F}_1\subset \overline{F}_{\ge p}$ implies $\dim_HF_1=\dim_H\underline{F}_1\le\dim_H\overline{F}_1\le \dim_H\overline{F}_{\ge p}$. Let $p\to 0$, by Lemma \ref{varlevelsets} (2), we get $\dim_HF_1=\dim_H\underline{F}_1=\dim_H\overline{F}_1=0$.
\end{proof}

We give the Hausdorff dimensions of these three kinds of level sets for a class of $\beta$.

\begin{theorem}\label{levelsetdimension}
Let $1<\beta<2$, $m\in\N_{\ge0}$ such that $\epsilon(1,\beta)=10^m10^\infty$.
\newline (1) If $0\le p<\frac{m+1}{m+2}$, then $F_p=\underline{F}_p=\overline{F}_p=\emptyset$ and $\dim_H F_p=\dim_H\underline{F}_p=\dim_H\overline{F}_p=0$.
\newline (2) If $\frac{m+1}{m+2}\le p\le 1$, then $\dim_HF_p=\dim_H\underline{F}_p=\dim_H\overline{F}_p$
\begin{small}
$$=\frac{(mp-m+p)\log(mp-m+p)-(mp-m+2p-1)\log(mp-m+2p-1)-(1-p)\log(1-p)}{\log\beta}.$$
\end{small}
In particular, $\dim_HF_{\frac{m+1}{m+2}}=\dim_H\underline{F}_{\frac{m+1}{m+2}}=\dim_H\overline{F}_{\frac{m+1}{m+2}}=\dim_HF_1=\dim_H\underline{F}_1=\dim_H\overline{F}_1=0$.
\end{theorem}

\begin{remark}\label{well-known} Take $m=0$ in Theorem \ref{levelsetdimension}. We get the well-known result (see for example \cite{FZ04})
$$\dim_HF_p=\frac{p\log p-(2p-1)\log(2p-1)-(1-p)\log(1-p)}{\log\beta}$$
where $\beta=\frac{\sqrt{5}+1}{2}$ is the golden ratio and $\frac{1}{2}\le p\le1$.
\end{remark}

\begin{proof}[Proof of Theorem \ref{levelsetdimension}]
\item(1) For any $x\in[0,1)$, by Lemma \ref{charADM}, each digit $1$ in $\epsilon(x,\beta)$ must be followed by at least $(m+1)$ consecutive $0$s. Thus
    $$\varlimsup_{n\to\infty}\frac{N_1(x,n)}{n}\le\frac{1}{m+2}\quad\text{and then}\quad\varliminf_{n\to\infty}\frac{\sharp\{1\le k\le n:\epsilon_k(x,\beta)=0\}}{n}\ge\frac{m+1}{m+2}$$
    for any $x\in[0,1)$. If $0\le p<\frac{m+1}{m+2}$, we get $F_p=\underline{F}_p=\overline{F}_p=\emptyset$.
\item(2) \textcircled{\footnotesize{$1$}} First, we consider $\frac{m+1}{m+2}< p< 1$.
\newline For any $x\in[1,0)$ and $n\in\N$, by Proposition \ref{cylinderlength}, we get
$$\frac{1}{n\log\beta-\log c}\le\frac{1}{-\log|I_n(x)|}\le\frac{1}{n\log\beta}.$$
Let $q:=\frac{mp-m+2p-1}{mp-m+p}$. Then $0<q<1$ since $\frac{m+1}{m+2}<p<1$. Let $\nu_q$ be the $(q,1-q)$ Bernoulli measure on $[0,1)$. It follows from
$$-\log\nu_q(I_n(x))=N_0(x,n)(-\log q)+N_1(x,n)(-\log(1-q))$$
that
\begin{align}\label{uplowbound}\frac{\frac{N_0(x,n)}{n}(-\log q)+\frac{N_1(x,n)}{n}(-\log(1-q))}{\log\beta-\frac{\log c}{n}}\le\frac{\log\nu_q(I_n(x))}{\log|I_n(x)|}\le\frac{\frac{N_0(x,n)}{n}(-\log q)+\frac{N_1(x,n)}{n}(-\log(1-q))}{\log\beta}.
\end{align}
Taking $\varliminf_{n\to\infty}$, we get
$$\underline{\dim}_{loc}^\beta\nu_q(x)=\varliminf_{n\to\infty}\frac{\frac{N_0(x,n)}{n}(-\log q)+\frac{N_1(x,n)}{n}(-\log(1-q))}{\log\beta}.$$
\begin{itemize}
\item[i)] Prove $\dim_H\underline{F}_p\le\frac{(1-(m+2)(1-p))(-\log q)+(1-p)(-\log(1-q))}{\log\beta}$.
  \newline For any $x\in \underline{F}_p$, we have $\varlimsup\limits_{n\to\infty}\frac{N_1(x,n)}{n}=1-p$ and then by Lemma \ref{n0estimate}, $\varliminf\limits_{n\to\infty}\frac{N_0(x,n)}{n}=1-(m+2)(1-p)$. Thus
  \begin{eqnarray*}
  \underline{\dim}_{loc}^\beta\nu_q(x)&\le&\frac{\varliminf\limits_{n\to\infty}\frac{N_0(x,n)}{n}(-\log q)+\varlimsup\limits_{n\to\infty}\frac{N_1(x,n)}{n}(-\log(1-q))}{\log\beta}\\
  &=&\frac{(1-(m+2)(1-p))(-\log q)+(1-p)(-\log(1-q))}{\log\beta}.
  \end{eqnarray*}
  Then we apply Theorem \ref{dimestimate} (1).
\item[ii)] Prove $\dim_H\overline{F}_p\le\frac{(1-(m+2)(1-p))(-\log q)+(1-p)(-\log(1-q))}{\log\beta}$.
  \newline For any $x\in \overline{F}_p$, we have $\varliminf\limits_{n\to\infty}\frac{N_1(x,n)}{n}=1-p$ and then by Lemma \ref{n0estimate}, $\varlimsup\limits_{n\to\infty}\frac{N_0(x,n)}{n}=1-(m+2)(1-p)$. Thus
  \begin{eqnarray*}
  \underline{\dim}_{loc}^\beta\nu_q(x)&\le&\frac{\varlimsup\limits_{n\to\infty}\frac{N_0(x,n)}{n}(-\log q)+\varliminf\limits_{n\to\infty}\frac{N_1(x,n)}{n}(-\log(1-q))}{\log\beta}\\
  &=&\frac{(1-(m+2)(1-p))(-\log q)+(1-p)(-\log(1-q))}{\log\beta}.
  \end{eqnarray*}
  Then we apply Theorem \ref{dimestimate} (1).
\item[iii)] Prove $\dim_HF_p\ge\frac{(1-(m+2)(1-p))(-\log q)+(1-p)(-\log(1-q))}{\log\beta}$.
  \newline For any $x\in F_p$, we have $\lim_{n\to\infty}\frac{N_1(x,n)}{n}=1-p$ and then by Lemma \ref{n0estimate}, $\lim_{n\to\infty}\frac{N_0(x,n)}{n}=1-(m+2)(1-p)$. Thus
  $$\underline{\dim}_{loc}^\beta\nu_q(x)=\frac{(1-(m+2)(1-p))(-\log q)+(1-p)(-\log(1-q))}{\log\beta}.$$
  By Theorem \ref{dimestimate} (2), it suffices to prove $\nu_q(F_p)=1>0$.
  \newline By $\epsilon_k(x,\beta)=0\Leftrightarrow\lfloor\beta T_\beta^{k-1}x\rfloor=0\Leftrightarrow0\le T_\beta^{k-1}x\le\frac{1}{\beta}\Leftrightarrow\mathbbm{1}_{[0,\frac{1}{\beta})}(T_\beta^{k-1}x)=1$, we get $$\frac{1}{n}\sharp\{1\le k\le n:\epsilon_k(x,\beta)=0\}=\frac{1}{n}\sum_{k=1}^n\mathbbm{1}_{[0,\frac{1}{\beta})}(T_\beta^{k-1}x).$$
  Since $([0,1),\mathcal{B}[0,1),m_q,T_\beta)$ is ergodic and the indicator function $\mathbbm{1}_{[0,\frac{1}{\beta})}$ is $m_q$-integrable, it follows from the Birkhoff Ergodic Theorem that
  $$\lim_{n\to\infty}\frac{1}{n}\sum_{k=1}^n\mathbbm{1}_{[0,\frac{1}{\beta})}(T_\beta^{k-1}x)=\int\mathbbm{1}_{[0,\frac{1}{\beta})}dm_q=m_q[0,\frac{1}{\beta})\xlongequal[\text{Lemma \ref{mpcylinder}}]{\text{by}}\frac{m(1-q)+1}{(m+1)(1-q)+1}\xlongequal[\text{def. of }q]{\text{by the}}p$$
  for $m_q\text{-}a.e.\text{ } x\in[0,1)$. Therefore $m_q(F_p)=1$. By $m_q\sim\nu_q$, we get $\nu_q(F_p)=1>0$.
\end{itemize}
Combining i), ii) iii) and $F_p=\underline{F}_p\cap\overline{F}_p$, we get
$$\dim_HF_p=\dim_H\underline{F}_p=\dim_H\overline{F}_p=\frac{(1-(m+2)(1-p))(-\log q)+(1-p)(-\log(1-q))}{\log\beta}.$$
We draw the conclusion by $q=\frac{mp-m+2p-1}{mp-m+p}$.
\newline\textcircled{\footnotesize{$2$}} For $p=1$, it follows from Theorem \ref{upbounddimention} that $\dim_HF_1=\dim_H\underline{F}_1=\dim_H\overline{F}_1=0$.
\newline\textcircled{\footnotesize{$3$}} Prove $\dim_HF_{\frac{m+1}{m+2}}=\dim_H\underline{F}_{\frac{m+1}{m+2}}=\dim_H\overline{F}_{\frac{m+1}{m+2}}=0$.
\newline By $\varliminf\limits_{n\to\infty}\frac{\sharp\{1\le k\le n:\epsilon_k(x,\beta)=0\}}{n}\ge\frac{m+1}{m+2}$ for any $x\in[0,1)$ in (1), we get $F_{\frac{m+1}{m+2}}=\overline{F}_{\frac{m+1}{m+2}}$. Since $F_{\frac{m+1}{m+2}}\subset\underline{F}_{\frac{m+1}{m+2}}$, it suffices to prove $\dim\underline{F}_{\frac{m+1}{m+2}}=0$.
\newline For $\frac{m+1}{m+2}<p<1$, let $q:=\frac{mp-m+2p-1}{mp-m+p}$. Then $0<q<1$. For any $x\in \underline{F}_{\le p}$ (see Lemma \ref{varlevelsets} (1) for definition), we have $\varlimsup\limits_{n\to\infty}\frac{N_1(x,n)}{n}\ge1-p$ and then by Lemma \ref{n0estimate}, $\varliminf_{n\to\infty}\frac{N_0(x,n)}{n}\le1-(m+2)(1-p)$. It follows from $\frac{N_1(x,n)}{n}\le1$ $(\forall n\in\N)$ and \eqref{uplowbound} that
$$\varliminf_{n\to\infty}\frac{\log\nu_q(I_n(x))}{\log|I_n(x)|}\le-\frac{(1-(m+2)(1-p))\log q+\log(1-q)}{\log\beta}$$
for any $x\in \underline{F}_{\le p}$. By Theorem \ref{dimestimate} (1) and the definition of $q$, we get
\begin{small}
$$\dim_H\underline{F}_{\le p}\le-\frac{(mp-m+2p-1)\log(mp-m+2p-1)-(mp-m+2p-1)\log(mp-m+p)+\log(1-q)}{\log\beta}.$$
\end{small}
For any $\frac{m+1}{m+2}<p<1$, $\underline{F}_{\frac{m+1}{m+2}}\subset \underline{F}_{\le p}$ implies $\dim_H\underline{F}_{\frac{m+1}{m+2}}\le\dim_H\underline{F}_{\le p}$. Let $p\to\frac{m+1}{m+2}$, then $q\to0$ and we get $\dim_H\underline{F}_{\frac{m+1}{m+2}}=0$.
\end{proof}

\begin{lemma}\label{n0estimate}
Let $1<\beta<2$ and $m\in\N_{\ge0}$ such that $\epsilon(1,\beta)=10^m10^\infty$. Then for any $x\in[0,1)$ and $n\ge m+2$, we have $n\le N_0(x,n)+(m+2)N_1(x,n)\le n+m+1$.
\end{lemma}
\begin{proof}
Let $w\in\Sigma_\beta^n$. It suffices to prove $n\overset{(1)}{\le}N_0(w)+(m+2)N_1(w)\overset{(2)}{\le} n+m+1$.
\item(1) Write
$$\cN_{10}(w):=\{2\le k\le n:w_{k-1}w_k=10\},\quad\cN_{100}(w):=\{3\le k\le n:w_{k-2}w_{k-1}w_k=100\},$$
$$\cdots,\quad\cN_{10^{m+1}}(w):=\{m+2\le k\le n:w_{k-m-1}\cdots w_k=10^{m+1}\}$$
and let
$$N_{10}(w):=\sharp\cN_{10}(w),\quad N_{100}(w):=\sharp\cN_{100}(w),\quad\cdots,\quad N_{10^{m+1}}(w):=\sharp\cN_{10^{m+1}}(w).$$
Noting that by Proposition \ref{tail-non-full}, $u0^{m+1}$ is full for any $u\in\Sigma_\beta^*$ and then $u0^{m+1}1$ is admissible, we get
$$\{1\le k\le n:w_k=0\}=(\cN_0(w)+1)\cup\cN_{10}(w)\cup\cN_{100}(w)\cup\cdots\cup\cN_{10^{m+1}}$$
which is a disjoint union. Thus
$$\sharp\{1\le k\le n:w_k=0\}=N_0(w)+N_{10}(w)+N_{100}(w)+\cdots+N_{10^{m+1}}(w)$$
and then
$$n=N_0(w)+N_{10}(w)+N_{100}(w)+\cdots+N_{10^{m+1}}(w)+N_1(w).$$
By $N_{10}(w), N_{100}(w), \cdots, N_{10^{m+1}}(w)\le N_1(w)$, we get $n\le N_0(w)+(m+2)N_1(w)$.
\item(2) If $N_1(w)=0$, the conclusion is obvious. If $N_1(w)\ge1$, except for the last digit $1$ in $w$, by Lemma \ref{charADM}, the other $1$s must be followed by at least (m+1) consecutive $0$s, and non of these $0$s can be replaced by $1$ to get an admissible word. Therefore
    $$N_1(w)+(m+1)(N_1(w)-1)+N_0(w)\le n,\quad i.e.,\quad N_0(w)+(m+2)N_1(w)\le n+m+1.$$
\end{proof}

\begin{lemma}\label{mpcylinder}
Let $1<\beta<2$, $m\in\N_{\ge0}$ such that $\epsilon(1,\beta)=10^m10^\infty$ and $0<p<1$. Then
$$m_p[0,\frac{1}{\beta})=\frac{m(1-p)+1}{(m+1)(1-p)+1}$$
where $m_p$ is given by Theorem \ref{mp}.
\end{lemma}
\begin{proof}
Notice that $m_p[0,\frac{1}{\beta})=1-m_p[\frac{1}{\beta},1)$ where
$$m_p[\frac{1}{\beta},1)=\lim_{n\to\infty}\frac{1}{n}\sum_{k=0}^{n-1}\nu_pT_\beta^{-k}[\frac{1}{\beta},1)=\lim_{n\to\infty}\frac{1}{n}\sum_{k=0}^{n-1}\mu_p\sigma_\beta^{-k}[1]$$
by Theorem \ref{mp}. For any $k\in\N_{\ge0}$, let
$$a_k:=\mu_p\sigma_\beta^{-k}[1]=\sum_{u_1\cdots u_k1\in\Sigma_\beta^*}\mu_p[u_1\cdots u_k1]$$
and
$$b_k:=\mu_p\sigma_\beta^{-k}[0^{m+1}]=\sum_{u_1\cdots u_k0^{m+1}\in\Sigma_\beta^*}\mu_p[u_1\cdots u_k0^{m+1}].$$
By Theorem \ref{mp}, the limits
$$a:=\lim_{n\to\infty}\frac{1}{n}\sum_{k=0}^{n-1}a_k\quad\text{and}\quad b:=\lim_{n\to\infty}\frac{1}{n}\sum_{k=0}^{n-1}b_k$$
exist.
\item(1) Prove $a=(1-p)b$. Write
\begin{eqnarray*}
b_{k+1}&=&\sum_{u_1\cdots u_k00^{m+1}\in\Sigma_\beta^*}\mu_p[u_1\cdots u_k00^{m+1}]+\sum_{u_1\cdots u_k10^{m+1}\in\Sigma_\beta^*}\mu_p[u_1\cdots u_k10^{m+1}]\\
&=&\sum_{u_1\cdots u_k0^{m+1}\in\Sigma_\beta^*}\mu_p[u_1\cdots u_k0^{m+1}0]+\sum_{u_1\cdots u_k1\in\Sigma_\beta^*}\mu_p[u_1\cdots u_k10^{m+1}].
\end{eqnarray*}
On the one hand, by Proposition \ref{tail-non-full}, $u_1\cdots u_k0^{m+1}$ is full and then $u_1\cdots u_k0^{m+1}1\in\Sigma_\beta^*$. On the other hand, by Lemma \ref{charADM}, for any $0\le s\le m$, $u_1\cdots u_k10^s10^{m-s}\notin\Sigma_\beta^*$ and then $[u_1\cdots u_k10^{m+1}]=[u_1\cdots u_k1]$. Thus, it follows from the definition of $\mu_p$ that
$$b_{k+1}=p\sum_{u_1\cdots u_k0^{m+1}\in\Sigma_\beta^*}\mu_p[u_1\cdots u_k0^{m+1}]+\sum_{u_1\cdots u_k1\in\Sigma_\beta^*}\mu_p[u_1\cdots u_k1]=pb_k+a_k.$$
Let $n\to\infty$ in
$$\frac{1}{n}\sum_{k=0}^{n-1}b_{k+1}=p\cdot\frac{1}{n}\sum_{k=0}^{n-1}b_k+\frac{1}{n}\sum_{k=0}^{n-1}a_k.$$
We get $b=pb+a$.
\item(2) Prove $b+(m+1)a=1$. It follows from
\begin{eqnarray*}
&\empty&\Big(\bigcup_{u_1\cdots u_k0^{m+1}\in\Sigma_\beta^*}[u_1\cdots u_k0^{m+1}]\Big)\cup\Big(\bigcup_{u_1\cdots u_k1\in\Sigma_\beta^*}[u_1\cdots u_k1]\Big)\\
&\cup&\Big(\bigcup_{u_1\cdots u_{k+1}1\in\Sigma_\beta^*}[u_1\cdots u_{k+1}1]\Big)\cup\cdots\cup\Big(\bigcup_{u_1\cdots u_{k+m}1\in\Sigma_\beta^*}[u_1\cdots u_{k+m}1]\Big)\\
&=&\Big(\bigcup_{u_1\cdots u_k0^{m+1}\in\Sigma_\beta^*}[u_1\cdots u_k0^{m+1}]\Big)\cup\Big(\bigcup_{u_1\cdots u_k10^m\in\Sigma_\beta^*}[u_1\cdots u_k10^m]\Big)\\
&\cup&\Big(\bigcup_{u_1\cdots u_{k+1}10^{m-1}\in\Sigma_\beta^*}[u_1\cdots u_{k+1}10^{m-1}]\Big)\cup\cdots\cup\Big(\bigcup_{u_1\cdots u_{k+m}1\in\Sigma_\beta^*}[u_1\cdots u_{k+m}1]\Big)\\
&=&\Sigma_\beta
\end{eqnarray*}
that $b_k+a_k+a_{k+1}+\cdots+a_{k+m}=1$. Let $n\to\infty$ in
$$\frac{1}{n}\sum_{k=0}^{n-1}b_k+\frac{1}{n}\sum_{k=0}^{n-1}a_k+\frac{1}{n}\sum_{k=0}^{n-1}a_{k+1}+\cdots+\frac{1}{n}\sum_{k=0}^{n-1}a_{k+m}=1.$$
We get $b+a+a+\cdots+a=1$.
\item(3) It follows from (1) and (2) that $a=\frac{1-p}{(m+1)(1-p)+1}$. Therefore
$$m_p[0,\frac{1}{\beta})=1-a=\frac{m(1-p)+1}{(m+1)(1-p)+1}.$$
\end{proof}

\section{Proofs of the examples}\label{sec:examples}

Let $\cM_\sigma(\Sigma_\beta)$ be the set of $\sigma$-invariant probability Borel measure on $(\Sigma_\beta,\cB(\Sigma_\beta))$ and $\cM_{T_\beta}([0,1))$ be the set of $T_\beta$-invariant probability Borel measure on $([0,1),\cB[0,1))$. We need the following.

\begin{definition}[$k$-step Markov measure] Let $k\in\N$ and $\mu\in\cM_\sigma(\Sigma_\beta)$. We call $\mu$ a \textit{$k$-step Markov measure} if there exists an $1\times2^k$ \textit{probability vector} $p=(p_{(i_1\cdots i_k)})_{i_1,\cdots,i_k=0,1}$ (i.e., $\sum_{i_1,\cdots,i_k=0,1}p_{(i_1\cdots i_k)}=1$ and $p_{(i_1\cdots i_k)}\ge0$ for all $i_1,\cdots,i_k\in\{0,1\}$) and a $2^k\times2^k$ \textit{stochastic matrix} $P=(P_{(i_1\cdots i_k)(j_1\cdots j_k)})_{i_1,\cdots,i_k,j_1,\cdots,j_k=0,1}$ (i.e., $\sum_{j_1,\cdots,j_k=0,1}P_{(i_1\cdots i_k)(j_1\cdots j_k)}=1$ for all $i_1,\cdots,i_k\in\{0,1\}$ and $P_{(i_1\cdots i_k)(j_1\cdots j_k)}\ge0$ for all $i_1,\cdots,i_k,$ $j_1,\cdots,j_k\in\{0,1\}$) with $pP=p$ such that
$$\mu[i_1\cdots i_k]=p_{(i_1\cdots i_k)}$$
for all $i_1,\cdots,i_k\in\{0,1\}$ and
$$\mu[i_1\cdots i_n]=p_{(i_1\cdots i_k)}P_{(i_1\cdots i_k)(i_2\cdots i_{k+1})}P_{(i_2\cdots i_{k+1})(i_3\cdots i_{k+2})}\cdots P_{(i_{n-k}\cdots i_{n-1})(i_{n-k+1}\cdots i_n)}$$
for all $i_1,\cdots,i_n\in\{0,1\}$ and $n>k$.
\end{definition}

We prove the following useful lemma for self-contained (see also \cite[Observation 6.2.7]{Ki97}).

\begin{lemma}\label{check Markov}
Let $k\ge1$ and $\mu\in\cM_\sigma(\Sigma_\beta)$. If
\begin{eqnarray}\label{Markov}
\frac{\mu[w_1\cdots w_{n+k+1}]}{\mu[w_1\cdots w_{n+k}]}=\frac{\mu[w_{n+1}\cdots w_{n+k+1}]}{\mu[w_{n+1}\cdots w_{n+k}]}
\end{eqnarray}
for all $w_1\cdots w_{n+k+1}\in\Sigma_\beta^{n+k+1}$ and $n\ge1$, then $\mu$ is a $k$-step Markov measure.
\end{lemma}
\begin{proof} For any $i_1,\cdots,i_k\in\{0,1\}$, let $p_{(i_1\cdots i_k)}:=\mu[i_1\cdots i_k]$. Then $p=(p_{(i_1\cdots i_k)})_{i_1,\cdots,i_k=0,1}$ is a $1\times2^k$ probability vector. We define a $2^k\times2^k$ stochastic matrix
$$P=(P_{(i_1\cdots i_k)(j_2\cdots j_{k+1})})_{i_1,\cdots,i_k,j_2,\cdots,j_{k+1}=0,1}$$
as follows.
\newline i) If there exists integer $t$ with $2\le t\le k$ such that $i_t\neq j_t$, let
$$P_{(i_1i_2\cdots i_k)(j_2\cdots j_kj_{k+1})}:=0;$$
\newline ii) If $\mu[i_1\cdots i_k]\neq0$, let
$$P_{(i_1\cdots i_k)(i_2\cdots i_{k+1})}:=\frac{\mu[i_1\cdots i_{k+1}]}{\mu[i_1\cdots i_k]}\quad \text{for } i_{k+1}=0,1;$$
\newline iii) If $\mu[i_1\cdots i_k]=0$, let
$$P_{(i_1\cdots i_k)(i_2\cdots i_k0)}:=1\quad\text{and}\quad P_{(i_1\cdots i_k)(i_2\cdots i_k1)}:=0.$$
Then $\sum_{j_2,\cdots,j_{k+1}=0,1}P_{(i_1\cdots i_{k})(j_2\cdots j_{k+1})}=1$ for all $i_1,\cdots,i_{k}\in\{0,1\}$ and $pP=p$. Since for all $s\ge1$ and $i_1,\cdots,i_{s+k}\in\{0,1\}$ we have
$$\begin{aligned}
\mu[i_1\cdots i_{k+s}]&=\mu[i_1\cdots i_k]\frac{\mu[i_1\cdots i_{k+1}]}{\mu[i_1\cdots i_k]}\frac{\mu[i_1i_2\cdots i_{k+2}]}{\mu[i_1i_2\cdots i_{k+1}]}\cdots\frac{\mu[i_1\cdots i_s\cdots i_{s+k}]}{\mu[i_1\cdots i_s\cdots i_{s+k-1}]}\\
&\xlongequal[(\ref{Markov})]{by}\mu[i_1\cdots i_k]\frac{\mu[i_1\cdots i_{k+1}]}{\mu[i_1\cdots i_k]}\frac{\mu[i_2\cdots i_{k+2}]}{\mu[i_2\cdots i_{k+1}]}\cdots\frac{\mu[i_s\cdots i_{s+k}]}{\mu[i_s\cdots i_{s+k-1}]}\\
&=p_{(i_1\cdots i_k)}P_{(i_1\cdots i_k)(i_2\cdots i_{k+1})}P_{(i_2\cdots i_{k+1})(i_3\cdots i_{k+2})}\cdots P_{(i_s\cdots i_{s+k-1})(i_{s+1}\cdots i_{s+k})},
\end{aligned}$$
by definition we know that $\mu$ is a $k$-step Markov measure.
\end{proof}

\begin{proof}[Proof of Example \ref{positive}]
Let $p\in(0,1)$ and $\lambda:=\lim\limits_{n\to\infty}\frac{1}{n}\sum_{k=0}^{n-1}\mu_p\circ\sigma^{-k}$. Then $\lambda$ is $\sigma$-invariant and $m_p=\lambda\circ\pi_\beta^{-1}$. It suffices to prove $$h_\lambda(\sigma)=\sup\Big\{h_\mu(\sigma):\mu\in\cM_\sigma(\Sigma_\beta)\text{ and }\mu[0]=\lambda[0]\Big\}.$$
Let $a:=\lambda[0]$. By \cite[Theorem 1.2 and Proposition 4.2]{Li19}, it suffices to prove that $\lambda$ is a unique $(m+1)$-step Markov measure (see \cite{Ki97,Li19} for definition) in $\cM_\sigma(\Sigma_\beta)$ taking value $a$ on $[0]$.
\newline(1) Prove the uniqueness. Noting that
\begin{eqnarray}\label{not admissible}
10^k1\notin\Sigma_\beta^*\text{ for all }0\le k\le m,
\end{eqnarray}
we get $\sigma^{-i}[1]=[0^i1]$ for all $0\le i\le m+1$. Let $\mu\in\cM_\sigma(\Sigma_\beta)$ with $\mu[0]=a$. Then we have
$$\mu[0^i1]=\mu[1]=1-a\quad\text{for all }0\le i\le m+1.$$
For $i,j\in\{0,1,\cdots,m+1\}$, by (\ref{not admissible}) we get $[0^i10^j]=[0^i1]$. Thus
$$\mu[0^i10^j]=\mu[0^i1]=1-a\quad\text{for all }0\le i,j\le m+1.$$
For $k\in\{1,\cdots,m+2\}$, also by (\ref{not admissible}) we get $\Sigma_\beta=[0^k]\cup\bigcup_{i=0}^{k-1}[0^i10^{k-i-1}]$, which implies
$$\mu[0^k]=1-\sum_{i=0}^{k-1}\mu[0^i10^{k-i-1}]=1-k(1-a)=ka-k+1.$$
The above calculation means that all the measures in $\cM_\sigma(\Sigma_\beta)$ taking value $a$ on $[0]$ are the same on all the cylinders with order no larger than $m+2$. Since $(m+1)$-step Markov measures only depend on their values on the cylinders with order no larger than $m+2$, the uniqueness of $\lambda$ follows.
\newline(2) Prove that $\lambda$ is an $(m+1)$-step Markov measure. Let $k:=m+1$. By Lemma \ref{check Markov}, it suffices to check (\ref{Markov}).
\newline\textcircled{\footnotesize{$1$}} For any $n\ge1$ and $w_1\cdots w_{n+k+1}\in\Sigma_\beta^{n+k+1}$, prove
$$\frac{\mu_p[w_1\cdots w_{n+k+1}]}{\mu_p[w_1\cdots w_{n+k}]}=\frac{\mu_p[w_{n+1}\cdots w_{n+k+1}]}{\mu_p[w_{n+1}\cdots w_{n+k}]}.$$
In fact, this follows from
$$\begin{aligned}
\frac{p^{N_0(w_1\cdots w_{n+k+1})}\cdot(1-p)^{N_1(w_1\cdots w_{n+k+1})}}{p^{N_0(w_1\cdots w_{n+k})}\cdot(1-p)^{N_1(w_1\cdots w_{n+k})}}&=p^{N_0(w_1\cdots w_{n+k+1})-N_0(w_1\cdots w_{n+k})}\cdot(1-p)^{N_1(w_{n+k+1})}\\
&\overset{(\star)}{=}p^{N_0(w_{n+1}\cdots w_{n+k+1})-N_0(w_{n+1}\cdots w_{n+k})}\cdot(1-p)^{N_1(w_{n+k+1})}\\
&=\frac{p^{N_0(w_{n+1}\cdots w_{n+k+1})}\cdot(1-p)^{N_1(w_{n+1}\cdots w_{n+k+1})}}{p^{N_0(w_{n+1}\cdots w_{n+k})}\cdot(1-p)^{N_1(w_{n+1}\cdots w_{n+k})}},
\end{aligned}$$
where $(\star)$ can be proved as follows. If $w_{n+k+1}=1$, then $(\star)$ is obviously true. If $w_{n+k+1}=0$, then
$$N_0(w_1\cdots w_{n+k+1})-N_0(w_1\cdots w_{n+k})=\left\{\begin{array}{ll}
1 & \mbox{if } w_1\cdots w_{n+k}1\in\Sigma_\beta^*\\
0 & \mbox{if } w_1\cdots w_{n+k}1\notin\Sigma_\beta^*
\end{array}\right.$$
and
$$N_0(w_{n+1}\cdots w_{n+k+1})-N_0(w_{n+1}\cdots w_{n+k})=\left\{\begin{array}{ll}
1 & \mbox{if } w_{n+1}\cdots w_{n+k}1\in\Sigma_\beta^*\\
0 & \mbox{if } w_{n+1}\cdots w_{n+k}1\notin\Sigma_\beta^*
\end{array}\right..$$
By $w_1\cdots w_{n+k}\in\Sigma_\beta^*$ and $\epsilon(1,\beta)=10^{k-1}10^\infty$, we know
$$w_1\cdots w_{n+k}1\in\Sigma_\beta^*\quad\Leftrightarrow\quad w_{n+1}\cdots w_{n+k}=0^k\quad\Leftrightarrow\quad w_{n+1}\cdots w_{n+k}1\in\Sigma_\beta^*.$$
Thus $N_0(w_1\cdots w_{n+k+1})-N_0(w_1\cdots w_{n+k})=N_0(w_{n+1}\cdots w_{n+k+1})-N_0(w_{n+1}\cdots w_{n+k})$.
\newline\textcircled{\footnotesize{$2$}} For any $n\ge1$ and $w_1\cdots w_{n+k+1}\in\Sigma_\beta^{n+k+1}$, prove
$$\frac{\mu_p\circ\sigma^{-1}[w_1\cdots w_{n+k+1}]}{\mu_p\circ\sigma^{-1}[w_1\cdots w_{n+k}]}=\frac{\mu_p\circ\sigma^{-1}[w_{n+1}\cdots w_{n+k+1}]}{\mu_p\circ\sigma^{-1}[w_{n+1}\cdots w_{n+k}]}=\frac{\mu_p[w_{n+1}\cdots w_{n+k+1}]}{\mu_p[w_{n+1}\cdots w_{n+k}]}.$$
By $w_1\cdots w_{n+k+1}\in\Sigma_\beta^*$ and $\epsilon(1,\beta)=10^{k-1}10^\infty$, we get
$$1w_{1}\cdots w_{n+k+1}\in\Sigma_\beta^*\quad\Leftrightarrow\quad w_{1}\cdots w_{k}=0^k\quad\Leftrightarrow\quad 1w_{1}\cdots w_{n+k}\in\Sigma_\beta^*,$$
which implies
$$\begin{aligned}\frac{\mu_p\circ\sigma^{-1}[w_{1}\cdots w_{n+k+1}]}{\mu_p\circ\sigma^{-1}[w_{1}\cdots w_{n+k}]}&=\left\{\begin{array}{ll}
\frac{\mu_p[0w_{1}\cdots w_{n+k+1}]+\mu_p[1w_{1}\cdots w_{n+k+1}]}{\mu_p[0w_{1}\cdots w_{n+k}]+\mu_p[1w_{1}\cdots w_{n+k}]} & \mbox{if } 1w_{1}\cdots w_{n+k+1}\in\Sigma_\beta^*\\
\frac{\mu_p[0w_{1}\cdots w_{n+k+1}]}{\mu_p[0w_{1}\cdots w_{n+k}]} & \mbox{if } 1w_{1}\cdots w_{n+k+1}\notin\Sigma_\beta^*
\end{array}\right.\\
&\xlongequal[]{\text{by \textcircled{\tiny{$1$}}}}\frac{\mu_p[w_{n+1}\cdots w_{n+k+1}]}{\mu_p[w_{n+1}\cdots w_{n+k}]}.
\end{aligned}$$
By $w_1\cdots w_{n+k+1}\in\Sigma_\beta^*$ and $\epsilon(1,\beta)=10^{k-1}10^\infty$, we get
$$1w_{n+1}\cdots w_{n+k+1}\in\Sigma_\beta^*\quad\Leftrightarrow\quad w_{n+1}\cdots w_{n+k}=0^k\quad\Leftrightarrow\quad 1w_{n+1}\cdots w_{n+k}\in\Sigma_\beta^*,$$
which implies
$$\begin{aligned}\frac{\mu_p\circ\sigma^{-1}[w_{n+1}\cdots w_{n+k+1}]}{\mu_p\circ\sigma^{-1}[w_{n+1}\cdots w_{n+k}]}&=\left\{\begin{array}{ll}
\frac{\mu_p[0w_{n+1}\cdots w_{n+k+1}]+\mu_p[1w_{n+1}\cdots w_{n+k+1}]}{\mu_p[0w_{n+1}\cdots w_{n+k}]+\mu_p[1w_{n+1}\cdots w_{n+k}]} & \mbox{if } 1w_{n+1}\cdots w_{n+k+1}\in\Sigma_\beta^*\\
\frac{\mu_p[0w_{n+1}\cdots w_{n+k+1}]}{\mu_p[0w_{n+1}\cdots w_{n+k}]} & \mbox{if } 1w_{n+1}\cdots w_{n+k+1}\notin\Sigma_\beta^*
\end{array}\right.\\
&\xlongequal[]{\text{by \textcircled{\tiny{$1$}}}}\frac{\mu_p[w_{n+1}\cdots w_{n+k+1}]}{\mu_p[w_{n+1}\cdots w_{n+k}]}.
\end{aligned}$$
\newline\textcircled{\footnotesize{$3$}} Repeat the above process. By induction, we can get that for any $j\ge0, n\ge1$ and $w_1\cdots w_{n+k+1}\in\Sigma_\beta^{n+k+1}$, we have
$$\frac{\mu_p\circ\sigma^{-j}[w_1\cdots w_{n+k+1}]}{\mu_p\circ\sigma^{-j}[w_1\cdots w_{n+k}]}=\frac{\mu_p\circ\sigma^{-j}[w_{n+1}\cdots w_{n+k+1}]}{\mu_p\circ\sigma^{-j}[w_{n+1}\cdots w_{n+k}]}=\frac{\mu_p[w_{n+1}\cdots w_{n+k+1}]}{\mu_p[w_{n+1}\cdots w_{n+k}]},$$
and then
$$\begin{aligned}
\frac{\lambda[w_1\cdots w_{n+k+1}]}{\lambda[w_1\cdots w_{n+k}]}&=\frac{\lim\limits_{s\to\infty}\frac{1}{s}\sum_{j=0}^{s-1}\mu_p\circ\sigma^{-j}[w_1\cdots w_{n+k+1}]}{\lim\limits_{s\to\infty}\frac{1}{s}\sum_{j=0}^{s-1}\mu_p\circ\sigma^{-j}[w_1\cdots w_{n+k}]}\\
&=\lim_{s\to\infty}\frac{\sum_{j=0}^{s-1}\mu_p\circ\sigma^{-j}[w_1\cdots w_{n+k+1}]}{\sum_{j=0}^{s-1}\mu_p\circ\sigma^{-j}[w_1\cdots w_{n+k}]}\\
&=\frac{\mu_p[w_{n+1}\cdots w_{n+k+1}]}{\mu_p[w_{n+1}\cdots w_{n+k}]}\\
&=\lim_{s\to\infty}\frac{\sum_{j=0}^{s-1}\mu_p\circ\sigma^{-j}[w_{n+1}\cdots w_{n+k+1}]}{\sum_{j=0}^{s-1}\mu_p\circ\sigma^{-j}[w_{n+1}\cdots w_{n+k}]}\\
&=\frac{\lim\limits_{s\to\infty}\frac{1}{s}\sum_{j=0}^{s-1}\mu_p\circ\sigma^{-j}[w_{n+1}\cdots w_{n+k+1}]}{\lim\limits_{s\to\infty}\frac{1}{s}\sum_{j=0}^{s-1}\mu_p\circ\sigma^{-j}[w_{n+1}\cdots w_{n+k}]}=\frac{\lambda[w_{n+1}\cdots w_{n+k+1}]}{\lambda[w_{n+1}\cdots w_{n+k}]}.
\end{aligned}$$
Therefore $\lambda$ satisfies (\ref{Markov}).
\end{proof}

\begin{proof}[Proof of Example \ref{counter}]
Let $p\in(0,1)$ and $\lambda:=\lim\limits_{n\to\infty}\frac{1}{n}\sum_{k=0}^{n-1}\mu_p\circ\sigma^{-k}$. Then $\lambda$ is $\sigma$-invariant and $m_p=\lambda\circ\pi_\beta^{-1}$. It suffices to prove $$h_\lambda(\sigma)<\sup\Big\{h_\mu(\sigma):\mu\in\cM_\sigma(\Sigma_\beta)\text{ and }\mu[0]=\lambda[0]\Big\}.$$
By the fact that $\cP:=\{[0],[1]\}$ is a partition generator of $\cB(\Sigma_\beta)$, we know $h_\lambda(\sigma)=h_\lambda(\sigma,\cP)$. Since $H_\lambda(\cP\big|\bigvee_{k=1}^n\sigma^{-k}\cP)$ decreases as $n$ increases, by \cite[Theorem 4.14]{Wal78} we get
$$h_\lambda(\sigma)\le H_\lambda\Big(\cP\Big|\bigvee_{k=1}^2\sigma^{-k}\cP\Big)$$
where
$$\begin{aligned}
H_\lambda\Big(\cP\Big|\sigma^{-1}\cP\bigvee\sigma^{-2}\cP\Big)&=H_\lambda\Big(\cP\Big|\sigma^{-1}\big(\cP\bigvee\sigma^{-1}\cP\big)\Big)\\
&=-\sum_{P\in\cP,Q\in\cP\bigvee\sigma^{-1}\cP}\lambda(P\cap\sigma^{-1}Q)\log\frac{\lambda(P\cap\sigma^{-1}Q)}{\lambda(\sigma^{-1}Q)}\\
&=-\sum_{i_1,i_2,i_3\in\{0,1\}}\lambda[i_1i_2i_3]\log\frac{\lambda[i_1i_2i_3]}{\lambda(\sigma^{-1}[i_2i_3])}\\
&=\sum_{i_1,i_2,i_3\in\{0,1\}}\lambda[i_1i_2i_3]\log\lambda[i_2i_3]-\sum_{i_1,i_2,i_3\in\{0,1\}}\lambda[i_1i_2i_3]\log\lambda[i_1i_2i_3]\\
&=\sum_{i_2,i_3\in\{0,1\}}\lambda[i_2i_3]\log\lambda[i_2i_3]-\sum_{i_1,i_2,i_3\in\{0,1\}}\lambda[i_1i_2i_3]\log\lambda[i_1i_2i_3]\\
&=\sum_{i_1,i_2\in\{0,1\}}\lambda[i_1i_2]\log\lambda[i_1i_2]-\sum_{i_1,i_2,i_3\in\{0,1\}}\lambda[i_1i_2i_3]\log\lambda[i_1i_2i_3]\\
&=\sum_{i_1,i_2,i_3\in\{0,1\}}\lambda[i_1i_2i_3]\log\lambda[i_1i_2]-\sum_{i_1,i_2,i_3\in\{0,1\}}\lambda[i_1i_2i_3]\log\lambda[i_1i_2i_3]\\
&=-\sum_{i_1,i_2,i_3\in\{0,1\}}\lambda[i_1i_2i_3]\log\frac{\lambda[i_1i_2i_3]}{\lambda[i_1i_2]},
\end{aligned}$$
where $0\log0$ is regarded as $0$. It follows from $[111]=\emptyset$ that $[110]=[11]$ and
$$\begin{aligned}
h_\lambda(\sigma)&\le-\sum_{\substack{i_1,i_2,i_3\in\{0,1\}\\i_1i_2\neq11}}\lambda[i_1i_2i_3]\log\frac{\lambda[i_1i_2i_3]}{\lambda[i_1i_2]}\\
&=-\sum_{i_1,i_3\in\{0,1\}}\lambda[i_10i_3]\log\frac{\lambda[i_10i_3]}{\lambda[i_10]}-\lambda[010]\log\frac{\lambda[010]}{\lambda[01]}-\lambda[011]\log\frac{\lambda[011]}{\lambda[01]}.
\end{aligned}$$
For $i_3\in\{0,1\}$, we have
$$\begin{aligned}
&-\lambda[00i_3]\log\frac{\lambda[00i_3]}{\lambda[00]}-\lambda[10i_3]\log\frac{\lambda[10i_3]}{\lambda[10]}\\
=&\lambda[0]\Big(\frac{\lambda[00]}{\lambda[0]}\big(-\frac{\lambda[00i_3]}{\lambda[00]}\log\frac{\lambda[00i_3]}{\lambda[00]}\big)+\frac{\lambda[10]}{\lambda[0]}\big(-\frac{\lambda[10i_3]}{\lambda[10]}\log\frac{\lambda[10i_3]}{\lambda[10]}\big)\Big)\\
\le&-\lambda[0i_3]\log\frac{\lambda[0i_3]}{\lambda[0]}
\end{aligned}$$
where the last inequality follows from Lemma \ref{inequality}. Thus
\begin{eqnarray}\label{le}
h_\lambda(\sigma)\le-\lambda[00]\log\frac{\lambda[00]}{\lambda[0]}-\lambda[01]\log\frac{\lambda[01]}{\lambda[0]}-\lambda[010]\log\frac{\lambda[010]}{\lambda[01]}-\lambda[011]\log\frac{\lambda[011]}{\lambda[01]}.
\end{eqnarray}
Since $\lambda$ is a $\sigma$-invariant probability measure, we have $\lambda[0]+\lambda[1]=1$, $\lambda[00]+\lambda[01]=\lambda[0]$, $\lambda[01]+\lambda[11]=\lambda[1]$, $\lambda[010]+\lambda[011]=\lambda[01]$ and $\lambda[011]+\lambda[111]=\lambda[11]$ where $\lambda[111]=0$. Let $a:=\lambda[0]$ and $b:=\lambda[01]$. Then by a simple calculation we get
$$\lambda[00]=a-b,\quad\lambda[010]=2b+a-1\quad\text{and}\quad\lambda[011]=1-a-b.$$
It follows from (\ref{le}) that
$$h_\lambda(\sigma)\le a\log a-(a-b)\log(a-b)-(1-a-b)\log(1-a-b)-(2b+a-1)\log(2b+a-1).$$
By Lemma \ref{mpcylinder1}, we know $a=\frac{p}{1-(1-p)^3}\ge\frac{1}{3}$. For $x\in[\frac{1-a}{2},\min\{a,1-a\}]$, we define the function
$$f_a(x):=a\log a-(a-x)\log(a-x)-(1-a-x)\log(1-a-x)-(2x+a-1)\log(2x+a-1).$$
Then $h_\lambda(\sigma)\le f_a(b)$. By calculating the derivative, it is straightforward to see that $f_a$ is strictly increasing on
$$\Big[\frac{1-a}{2},\frac{3-4a+\sqrt{-8a^2+12a-3}}{6}\Big]$$
and strictly decreasing on
$$\Big[\frac{3-4a+\sqrt{-8a^2+12a-3}}{6},\min\{a,1-a\}\Big].$$
By Lemma \ref{mpcylinder1}, it is not difficult to check $b\neq\frac{3-4a+\sqrt{-8a^2+12a-3}}{6}$. Thus $h_\lambda(\sigma)<\max f_a(x)$. By \cite[Proposition 4.2 and Remark 1.4]{Li19}, we have
$$\max f_a(x)=\sup\Big\{h_\mu(\sigma):\mu\in\cM_\sigma(\Sigma_\beta),\mu[0]=a\Big\}.$$
Therefore
$$h_\lambda(\sigma)<\sup\Big\{h_\mu(\sigma):\mu\in\cM_\sigma(\Sigma_\beta)\text{ and }\mu[0]=a\Big\}.$$
\end{proof}

The following lemma follows immediately from the convexity of the function $x\log x$.

\begin{lemma}\label{inequality}
Let $\phi:[0,\infty)\to\R$ be defined by
$$\phi(x)=\left\{\begin{array}{ll}
0 & \mbox{if } x=0;\\
-x\log x & \mbox{if } x>0.
\end{array}\right.$$
Then for all $x,y\in[0,\infty)$ and $a,b\ge0$ with $a+b=1$,
$$a\phi(x)+b\phi(y)\le\phi(ax+by).$$
The equality holds if and only if $x=y$, $a=0$ or $b=0$.
\end{lemma}

\begin{lemma}\label{mpcylinder1}
Let $\beta\in(1,2)$ be a pseudo-golden ratio, i.e., $\epsilon(1,\beta)=1^m0^\infty$ for some $m\in\N_{\ge2}$ and $0<p<1$. Then
$$m_p[0,\frac{1}{\beta})=\frac{p}{1-(1-p)^m}\quad\text{and}\quad m_p[\frac{1}{\beta}+\cdots+\frac{1}{\beta^{m-1}},1)=\frac{p(1-p)^{m-1}}{1-(1-p)^m}.$$
\end{lemma}
\begin{proof}
\item(1) By Theorem \ref{mp}, we get
$$m_p[0,\frac{1}{\beta})=\lim_{n\to\infty}\frac{1}{n}\sum_{k=0}^{n-1}\nu_pT_\beta^{-k}[0,\frac{1}{\beta})=\lim_{n\to\infty}\frac{1}{n}\sum_{k=0}^{n-1}\mu_p\sigma_\beta^{-k}[0].$$
For any $k\ge0$, it follows from
$$\Sigma_\beta^*=\{w\in\bigcup_{n=1}^\infty\{0,1\}^n: 1^m\text{ does not appear in } w\}$$
that
\begin{eqnarray*}
\Sigma_\beta&=&\bigcup_{u_1\cdots u_{k+m}\in\newline\Sigma_\beta^*}[u_1\cdots u_{k+m}]\\
&=&\Big(\bigcup_{u_1\cdots u_{k+m-1}\in\Sigma_\beta^*}[u_1\cdots u_{k+m-1}0]\Big)\cup\Big(\bigcup_{u_1\cdots u_{k+m-2}\in\Sigma_\beta^*}[u_1\cdots u_{k+m-2}01]\Big)\\
&\empty&\cup\Big(\bigcup_{u_1\cdots u_{k+m-3}\in\Sigma_\beta^*}[u_1\cdots u_{k+m-3}01^2]\Big)\cup\cdots\cup\Big(\bigcup_{u_1\cdots u_{k}\in\Sigma_\beta^*}[u_1\cdots u_{k}01^{m-1}]\Big)
\end{eqnarray*}
and then
\begin{eqnarray*}
1&=&\mu_p\sigma_\beta^{-(k+m-1)}[0]+\sum_{u_1\cdots u_{k+m-2}\in\Sigma_\beta^*}\mu_p[u_1\cdots u_{k+m-2}01]\\
&\empty&+\sum_{u_1\cdots u_{k+m-3}\in\Sigma_\beta^*}\mu_p[u_1\cdots u_{k+m-3}01^2]+\cdots+\sum_{u_1\cdots u_{k}\in\Sigma_\beta^*}\mu_p[u_1\cdots u_{k}01^{m-1}]\\
&=&\mu_p\sigma_\beta^{-(k+m-1)}[0]+(1-p)\sum_{u_1\cdots u_{k+m-2}\in\Sigma_\beta^*}\mu_p[u_1\cdots u_{k+m-2}0]\\
&\empty&+(1-p)^2\sum_{u_1\cdots u_{k+m-3}\in\Sigma_\beta^*}\mu_p[u_1\cdots u_{k+m-3}0]+\cdots+(1-p)^{m-1}\sum_{u_1\cdots u_{k}\in\Sigma_\beta^*}\mu_p[u_1\cdots u_{k}]\\
&=&\mu_p\sigma_\beta^{-(k+m-1)}[0]+(1-p)\mu_p\sigma_\beta^{-(k+m-2)}[0]+\cdots+(1-p)^{m-1}\mu_p\sigma_\beta^{-k}[0]
\end{eqnarray*}
Thus
$$\frac{1}{n}\sum_{k=0}^{n-1}\mu_p\sigma_\beta^{-(k+m-1)}[0]+(1-p)\frac{1}{n}\sum_{k=0}^{n-1}\mu_p\sigma_\beta^{-(k+m-2)}[0]+\cdots+(1-p)^{m-1}\frac{1}{n}\sum_{k=0}^{n-1}\mu_p\sigma_\beta^{-k}[0]=1.$$
Taking $n\to\infty$, we get
$$m_p[0,\frac{1}{\beta})+(1-p)m_p[0,\frac{1}{\beta})+\cdots+(1-p)^{m-1}m_p[0,\frac{1}{\beta})=1.$$
Therefore $m_p[0,\frac{1}{\beta})=\frac{p}{1-(1-p)^m}$.
\item(2) By Theorem \ref{mp}, we get
$$m_p[\frac{1}{\beta}+\cdots+\frac{1}{\beta^{m-1}},1)=\lim_{n\to\infty}\frac{1}{n}\sum_{k=0}^{n-1}\nu_pT_\beta^{-k}[\frac{1}{\beta}+\cdots+\frac{1}{\beta^{m-1}},1)=\lim_{n\to\infty}\frac{1}{n}\sum_{k=0}^{n-1}\mu_p\sigma_\beta^{-k}[1^{m-1}].$$
For any $k\ge0$, it follows from
$$\sigma_\beta^{-(k+1)}[1^{m-1}]=\bigcup_{u_1\cdots u_k u_{k+1}1^{m-1}\in\Sigma_\beta^*}[u_1\cdots u_k u_{k+1}1^{m-1}]=\bigcup_{u_1\cdots u_k\in\Sigma_\beta^*}[u_1\cdots u_k 01^{m-1}]$$
that
\begin{eqnarray*}
\mu_p\sigma_\beta^{-(k+1)}[1^{m-1}]&=&\sum_{u_1\cdots u_k\in\Sigma_\beta^*}\mu_p[u_1\cdots u_k01^{m-1}]\\
&=&(1-p)^{m-1}\sum_{u_1\cdots u_k\in\Sigma_\beta^*}\mu_p[u_1\cdots u_k0]\\
&=&(1-p)^{m-1}\mu_p\sigma_\beta^{-k}[0].
\end{eqnarray*}
Thus
$$\frac{1}{n}\sum_{k=0}^{n-1}\mu_p\sigma_\beta^{-(k+1)}[1^{m-1}]=\frac{1}{n}\sum_{k=0}^{n-1}(1-p)^{m-1}\mu_p\sigma_\beta^{-k}[0].$$
Taking $n\to\infty$, we get
$$m_p[\frac{1}{\beta}+\cdots+\frac{1}{\beta^{m-1}},1)=(1-p)^{m-1}m_p[0,\frac{1}{\beta})=\frac{p(1-p)^{m-1}}{1-(1-p)^m}.$$
\end{proof}

\begin{ack}
The first and the third author thank Tony Samuel for useful discussions. The first author was supported by NSFC 11671151 and Guangdong Natural Science Foundation 2018B0303110005. The second author is grateful to the Oversea Study Program of Guangzhou Elite Project.
\end{ack}


\begin{thebibliography}{00}

\bibitem[1]{BuWa14} \textsc{Y. Bugeaud and B.-W. Wang}, \textit{ Distribution of full cylinders and the Diophantine properties of the orbits in $\beta$-expansions}, J. Fractal Geom. 1 (2014), no. 2, 221-241.

\bibitem[2]{Co80} \textsc{D. L. Cohn}, \textit{ Measure Theory}, Birkh\"auser, 1980.

\bibitem[3]{Dud02} \textsc{R. M. Dudley}, \textit{ Real analysis and probability}, Cambridge University Press, Cambridge, UK, 2002.

\bibitem[4]{DK02} \textsc{K. Dajani and C. Kraaikamp}, \textit{Ergodic theory of numbers}, Carus Mathematical Monographs, 29. Mathematical Association of America, Washington, DC, 2002.

\bibitem[5]{DM46} \textsc{N. Dunford and D. S. Miller}, \textit{ On the ergodic theorem}, Trans. Amer. Math. Soc. 60, (1946), 538-549.

\bibitem[6]{Fal90} \textsc{K. J. Falconer}, \textit{Fractal Geometry - Mathematical Foundations and Applications}, John Wiley, 1990.

\bibitem[7]{Fal97} \textsc{K. J. Falconer}, \textit{Techniques in Fractal Geometry}, JohnWiley $\&$ Sons, Ltd, Chichester, 1997.

\bibitem[8]{FW12} \textsc{A.-H. Fan and B.-W. Wang}, \textit{ On the lengths of basic intervals in beta expansions}, Nonlinearity 25 (2012), no. 5, 1329-1343.

\bibitem[9]{FZ04} \textsc{A. Fan, H. Zhu}, \textit{ Level sets of $\beta$-expansions}, C. R. Acad. Sci. Paris, Ser. I 339 (2004).

\bibitem[10]{Ki97} \textsc{B. P. Kitchens}, \textit{Symbolic Dynamics: One-sided, Two-sided and Countable State Markov Shifts}, Springer Science $\&$ Business Media, 1997.

\bibitem[11]{Li19} \textsc{Y.-Q. Li}, \textit{ Hausdorff dimension of frequency sets in beta-expansions}, arXiv:1905.01481v2 (2019).

\bibitem[12]{LiLi18} \textsc{Y.-Q. Li and B. Li}, \textit{ Distributions of full and non-full words in beta-expansions}, Journal of Number Theory 190C (2018) pp. 311-332.

\bibitem[13]{LiWu08} \textsc{B. Li and J. Wu}, \textit{ Beta-expansion and continued fraction expansion}, J. Math. Anal. Appl. 339 (2008), no. 2, 1322-1331.

\bibitem[14]{Pa60} \textsc{W. Parry}, \textit{ On the $\beta$-expansions of real numbers}, Acta Math. Acad. Sci. Hungar. 11 (1960) 401-416.

\bibitem[15]{PfSu05} \textsc{C.-E. Pfister and W.G. Sullivan,} \textit{ Large deviations estimates for dynamical systems without the specification property. Applications to the $\beta$-shifts}, Nonlinearity, 18 (2005), 237--261.

\bibitem[16]{Ren57} \textsc{A. R\'enyi}, \textit{ Representations for real numbers and their ergodic properties}, Acta Math. Acad. Sci. Hungar 8 (1957), 477-493.

\bibitem[17]{Wal78} \textsc{P. Walters}, \textit{ Equilibrium States for $\beta$-Transformation and Related Transformations}, Math. Z. 159 (1978), no. 1, 65-88.

\bibitem[18]{Wal82} \textsc{P. Walters}, \textit{ An Introduction to Ergodic Theory}, Springer-Verlag, New York, Heidelberg, Berlin, 1982.

\bibitem[19]{Ward} \textsc{R. Ward}, \textit{On Robustness Properties of Beta Encoders and Golden Ratio Encoders}, IEEE Transactions on Information Theory, 54 (9): 4324--4334, 2008, arXiv:0806.1083, doi:10.1109/TIT.2008.928235

\end{thebibliography}
\end{document}